\documentclass[a4paper]{article}
\usepackage{amsmath,amsthm,amsfonts,thmtools}
\usepackage{amssymb}
\usepackage[cmintegrals]{newtxmath}
\usepackage{graphicx}
\usepackage{listings}
\usepackage[font=footnotesize]{caption}
\usepackage[font=scriptsize]{subcaption}
\usepackage{euscript}
\usepackage{units}
\usepackage{enumerate}
\usepackage{enumitem}
\usepackage{breqn}
\usepackage{xcolor}
\usepackage{todonotes}
\usepackage{booktabs}
\usepackage{tikz}
\usepackage{comment}
 
\DeclareMathAlphabet{\pazocal}{OMS}{zplm}{m}{n}
\usepackage[bookmarks=true,hidelinks]{hyperref}
\usepackage{bbm}
\usepackage{multirow}
\usepackage{algpseudocode}  
\usepackage{mathtools}
\usepackage{subcaption} 
\usepackage{graphicx}
\usepackage{pgfplots}
\usepackage{float}
\usepackage{algpseudocode}
\usepackage{bbold}
\usepackage{cite}
\usepackage{microtype}
\usepackage{rotating}
\usepackage{enumitem}
\usepackage{bm}
\usepackage{hyperref}
\providecommand{\keywords}[1]
{
  \small	
  \textbf{\textit{Keywords:}} #1
}
\numberwithin{equation}{section}
\newtheorem{theorem}{Theorem}[section]

\newtheorem{algorithm}[theorem]{Algorithm}

\newtheorem{appendixthm}{}[section] 
\numberwithin{equation}{section}

\theoremstyle{appendixstyle}

\newtheoremstyle{appendixstyle} 
  {3pt}   
  {3pt}   
  {\normalfont} 
  {}      
  {\bfseries} 
  {.}     
  {.5em}  
  {\thmnumber{#2}} 

\theoremstyle{definition}

\newtheoremstyle{myremarkstyle}{}{}{}{}{\bfseries}{.}{ }{}
\theoremstyle{myremarkstyle}
\declaretheorem[name=Remark,qed=$\blacksquare$,numberlike=theorem]{remark}

\makeatletter
\newcommand*{\intavg}{%
  \mint@l{-}{}%
}
\newcommand*{\mint@l}[2]{%
  \@ifnextchar\limits{%
    \mint@l{#1}%
  }{%
    \@ifnextchar\nolimits{%
      \mint@l{#1}%
    }{%
      \@ifnextchar\displaylimits{%
        \mint@l{#1}%
      }{%
        \mint@s{#2}{#1}%
      }%
    }%
  }%
}
\newcommand*{\mint@s}[2]{%
  \@ifnextchar_{%
    \mint@sub{#1}{#2}%
  }{%
    \@ifnextchar^{%
      \mint@sup{#1}{#2}%
    }{%
      \mint@{#1}{#2}{}{}%
    }%
  }%
}
\def\mint@sub#1#2_#3{%
  \@ifnextchar^{%
    \mint@sub@sup{#1}{#2}{#3}%
  }{%
    \mint@{#1}{#2}{#3}{}%
  }%
}
\def\mint@sup#1#2^#3{%
  \@ifnextchar_{%
    \mint@sub@sup{#1}{#2}{#3}%
  }{%
    \mint@{#1}{#2}{}{#3}%
  }%
}
\def\mint@sub@sup#1#2#3^#4{%
  \mint@{#1}{#2}{#3}{#4}%
}
\def\mint@sup@sub#1#2#3_#4{%
  \mint@{#1}{#2}{#4}{#3}%
}
\newcommand*{\mint@}[4]{%
  \mathop{}%
  \mkern-\thinmuskip
  \mathchoice{%
    \mint@@{#1}{#2}{#3}{#4}%
        \displaystyle\textstyle\scriptstyle
  }{%
    \mint@@{#1}{#2}{#3}{#4}%
        \textstyle\scriptstyle\scriptstyle
  }{%
    \mint@@{#1}{#2}{#3}{#4}%
        \scriptstyle\scriptscriptstyle\scriptscriptstyle
  }{%
    \mint@@{#1}{#2}{#3}{#4}%
        \scriptscriptstyle\scriptscriptstyle\scriptscriptstyle
  }%
  \mkern-\thinmuskip
  \int#1%
  \ifx\\#3\\\else_{#3}\fi
  \ifx\\#4\\\else^{#4}\fi  
}
\newcommand*{\mint@@}[7]{%
  \begingroup
    \sbox0{$#5\int\m@th$}%
    \sbox2{$#5\int_{}\m@th$}%
    \dimen2=\wd0 %
    \let\mint@limits=#1\relax
    \ifx\mint@limits\relax
      \sbox4{$#5\int_{\kern1sp}^{\kern1sp}\m@th$}%
      \ifdim\wd4>\wd2 %
        \let\mint@limits=\nolimits
      \else
        \let\mint@limits=\limits
      \fi
    \fi
    \ifx\mint@limits\displaylimits
      \ifx#5\displaystyle
        \let\mint@limits=\limits
      \fi
    \fi
    \ifx\mint@limits\limits
      \sbox0{$#7#3\m@th$}%
      \sbox2{$#7#4\m@th$}%
      \ifdim\wd0>\dimen2 %
        \dimen2=\wd0 %
      \fi
      \ifdim\wd2>\dimen2 %
        \dimen2=\wd2 %
      \fi
    \fi
    \rlap{%
      $#5%
        \vcenter{%
          \hbox to\dimen2{%
            \hss
            $#6{#2}\m@th$%
            \hss
          }%
        }%
      $%
    }%
  \endgroup
}

\def\XXint#1#2#3{{\setbox0=\hbox{$#1{#2#3}{\int}$ }
		\vcenter{\hbox{$#2#3$ }}\kern-.6\wd0}}

\renewcommand{\geq}{\geqslant}
\renewcommand{\leq}{\leqslant}

\renewcommand{\epsilon}{\varepsilon}
\renewcommand{\phi}{\varphi}

\newcommand{\train}{\mathcal{S}}

\begin{document}

\title{A residual weighted physics informed neural network for forward and inverse problems of reaction diffusion equations.}

\author{K. Murari   \footnotemark[2]  \and P. Roul  \footnotemark[3]
	\and S. Sundar \footnotemark[2]
}

\date{\today}

\maketitle
\medskip
\centerline{$\dagger$ Centre for Computational Mathematics and Data Science}
\centerline{Department of Mathematics, IIT Madras, Chennai 600036, India}
\centerline{$\ddagger$ Department of Mathematics, VNIT, Nagpur, Maharashtra 440010, India}
\centerline{Email:kmurari2712@gmail.com, drpkroul@mth.vnit.ac.in,  slnt@smail.iitm.ac.in }

 \begin{abstract} 
 In this work, we propose the Residual-Weighted Physics-Informed Neural Network (RW-PINN), a new method designed to enhance the accuracy of Physics-Informed Neural Network (PINN) -based algorithms.
We construct a deep learning framework with two residual-weighting schemes to solve reaction–diffusion equations and evaluate its performance on both forward and inverse problems. The approach computes weights proportional to the PDE residuals, rescales them, and incorporates these scaled residuals into the loss function, leading to more stable training. Furthermore, we establish generalized error bounds that account for training and quadrature errors, and we analyze the convergence and stability of the method. The proposed algorithms are validated through numerical experiments on nonlinear equations, supported by statistical error analysis. To further demonstrate the effectiveness of our methodology, we implemented PINN-based forward  and inverse frameworks for the nonlinear equations and conducted a comparative analysis with the proposed RW-PINN approach.
\end{abstract}

\medskip


\section{Introduction}\label{sec:intro}
\noindent

Reaction–diffusion equations have been widely applied across various fields and play a crucial role in modeling numerous real-world phenomena. In this work, we focus on two specific types of reaction–diffusion equations for detailed analysis: the Burgess equation and extended Fisher–Kolmogorov equation.
Burgess et al. \cite{burgess1997interaction} developed mathematical framework for gliomas in 1997. A model incorporating fractional operators was later developed by \cite{ganji2021mathematical}. More recently, \cite{nayied2023numerical} conducted simulations of this model using Fibonacci and Haar wavelets. The EFK equation was developed by augmenting the Fisher–Kolmogorov (FK) model with a fourth-order derivative term, as presented in \cite{coullet1987nature}, \cite{Dee1988}, and \cite{Van1987}. The EFK equation is widely applied across various physics disciplines, including fluid dynamics, plasma science, nuclear reactions, ecological modeling, and epidemic studies \cite{Ahlers1983}. However, these equations can exhibit very complex behavior, especially in reaction-diffusion systems, due to their nonlinear nature and complex computational domains \cite{Ahlers1983}. The work by \cite{danumjaya2006numerical} presents a numerical analysis of the finite element method applied to the EFK equation.

 A crucial application of the Fisher–Kolmogorov equation lies in the modeling of brain tumor dynamics \cite{ilati2020analysis}. The study employs the interpolating element-free Galerkin (IEFG) method for numerical simulation, offering a meshless approach that effectively handles complex tumor growth patterns. Various methods have been developed for the EFK equation, including the interpolating element-free Galerkin (IEFG) method \cite{ilati2020analysis}, finite difference and second-order schemes for 2D FK equations \cite{khiari2011finite}, \cite{kadri2011second}, and a Fourier pseudo-spectral method \cite{liu2017fourier}. The direct local boundary integral equation method was applied by \cite{ilati2018direct}, and an error analysis of IEFG was conducted by \cite{abbaszadeh2020error}. Meshfree schemes using radial basis functions were introduced by \cite{kumar2022radial}, and a meshless generalized finite difference method was developed by \cite{ju2023three}. Recent developments include adaptive low-rank splitting \cite{zhao2024adaptive}, finite element analysis \cite{al2024finite}, and superconvergence analysis of FEMs \cite{pei2024unconditional}.

Deep learning has become an essential technique for addressing the curse of dimensionality, making it a critical tool in modern technology and cutting-edge research over recent years. Deep learning techniques are particularly well-suited for approximating highly nonlinear functions by employing multiple layers of transformations and nonlinear functions. These methods, advanced statistical learning, and large-scale optimization techniques have become increasingly reliable for solving nonlinear and high-dimensional partial differential equations (PDEs). The universal approximation theorem, demonstrated by Cybenko \cite{cybenko1989approximations}, Hornik et al. \cite{hornik1989multilayer}, Barron \cite{barron1993universal}, and Yarotsky \cite{yarotsky2017error}, shows that deep neural networks (DNNs) can approximate any continuous function under specific conditions. This makes DNNs highly suitable for use as trial functions in solving PDEs. One common technique involves minimizing the residual of the PDE by evaluating it at discrete points, often called collocation points. Several algorithms have been developed based on deep learning, with some of the most prominent being Physics-Informed Neural Networks (PINNs) and deep operator networks such as Deeponets and their variants. PINNs, first introduced by Raissi et al. \cite{raissi2019physics}, have proven highly effective in addressing high-dimensional PDEs. Their mesh-free nature and ability to solve forward and inverse problems within a single optimization framework make them particularly powerful.  Extensions to this algorithm have been proposed in works such as \cite{jagtap2020extended}, \cite{jagtap2020conservative}, \cite{shukla2021parallel}, \cite{moseley2023finite} and \cite{yu2022gradient}, with libraries such as \cite{lu2021deepxde} developed to facilitate solving PDEs using PINNs. Furthermore, domain decomposition methods have been applied to PINNs by \cite{dolean2022finite}. Despite their success, challenges remain in training these models, as highlighted by \cite{wang2022and}, who explored these difficulties using the Neural Tangent Kernel (NTK) framework.
\cite{peng2022pinn} applied the PINN deep learning method to solve the Chen-Lee-Liu equation, focusing on rogue waves on a periodic background. \cite{pu2023data} used a PINN approach with parameter regularization to solve forward and inverse problems for the Yajima Oikawa system. Mishra and Molinaro analyzed the generalization error of PINN-based Algorithms for forward \cite{mishra2023estimates} and inverse problems \cite{mishra2022estimates} across various nonlinear PDEs. Mishra and his collaborators also derived error bounds \cite{de2022error}, \cite{de2024error}, \cite{mishra2021physics}, \cite{bai2021physics} and also introduced weak PINNs (wPINNs) for estimating entropy solutions to scalar conservation laws \cite{de2024wpinns}. \cite{bai2021physics} conducted numerical experiments and derived generalized error bounds for nonlinear dispersive equations, including the KdV-Kawahara, Camassa-Holm, and Benjamin-Ono equations, using PINN-based algorithms. The study by \cite{Zhang2023} examines the boundedness and convergence properties of neural networks in the context of PINNs. \cite{berrone2022variational} has introduced a variational formulation within the PINNs framework. The recent work of \cite{DeRyck2024} explores the numerical analysis of PINNs. Recent studies have introduced a range of promising deep learning approaches, as seen in \cite{wang2024dcem}, \cite{eshaghi2025variational}, \cite{Bai2024}, \cite{Noorizadegan2024}, \cite{Eshaghi2025} and \cite{Sun2023}. Recently, \cite{liu2024discontinuity} proposed a weighted treatment within the PINNs framework by introducing equation-dependent weighting strategies to enhance discontinuity resolution. \cite{rodrigues2024} has applied PINN for tumor cell growth modeling using  differential equation for montroll growth model, verhulst growth model. \cite{zhang2025personalized} determined individualized parameters for a reaction-diffusion PDE framework describing glioblastoma progression using a single 3D structural MRI scan. \cite{Zapf2022} analyzed the movement of molecules within the human brain using MRI data and PINNs. \\

The contribution of the work is following:
This study presents a deep learning framework for solving the Burgess and extended Fisher–Kolmogorov equations, which describe glioblastoma progression, in both forward and inverse problem settings. We introduce a Residual-Weighted Physics-Informed Neural Network (RW-PINN), which
 leverages residual-based scaling to significantly improve the accuracy of PINN solutions. The RW-PINN architecture is developed to yield precise numerical approximations for the reaction diffusion equations. The residual and the corresponding loss function approximation are derived. The proposed approach establishes a strong theoretical foundation by formulating a rigorous generalized error bound, which is expressed in terms of training and quadrature errors. Additionally, a rigorous proof of the boundedness and convergence of the neural network is provided, verifying the theoretical validity of the neural network approximation. Numerical experiments are conducted for both forward and inverse problems in  nonlinear cases. Extensive computational results, supported by statistical analyses, demonstrate the method’s effectiveness and accuracy. Furthermore, we applied PINN-based forward~\cite{mishra2023estimates,bai2021physics} and inverse~\cite{mishra2022estimates} frameworks to reaction diffusion equations and carried out a comparative study with the proposed methods.
The results demonstrate that RW-PINN-based algorithms serve as effective and robust tools for simulating nonlinear PDEs, providing a reliable computational framework for these equations.\\
\textbf{Novelty:} \textit{In this work, we investigate the simulation of the Burgess and extended Fisher–Kolmogorov equations by leveraging deep learning–based algorithms for accurate solution approximation. We propose a new method, termed the Residual-Weighted Physics-Informed Neural Network (RW-PINN), to enhance the accuracy and robustness of  PINN-based forward and inverse algorithms. In this approach, weights are assigned proportionately to the PDE residuals, appropriately scaled, and subsequently incorporated into the loss function, facilitating more efficient and stable convergence. Specifically, two distinct weighting strategies based on the PDE residual are considered: (i) sigmoid-based weighting and (ii) softplus-based weighting, both of which are discussed in detail in the residual section. Such approaches represent a valuable contribution to the class of residual-based methods. These approaches have not been previously reported in the literature}.

This paper is structured as follows: Section \ref{sec:2} presents the mathematical formulation and methodology, including the RW-PINN framework, governing equations, quadrature techniques, neural network design, residual computation, loss functions, optimization approach, generalization error estimation, and the stability and convergence of multilayer neural networks. Section \ref{sec:3} details numerical experiments and validates the proposed approach. Section \ref{sec:4} provides a theoretical measure of errors. Finally, Section \ref{sec:5} summarizes the key findings. An appendix is included for proofs and lemmas.
\section{RW-PINN Approximation}\label{sec:2}
The RW-PINN approximation for reaction–diffusion equations is formulated by defining the model architecture, introducing residual-weighting schemes, and embedding them within the PINN framework.

\subsection{Models}
The two different types of reaction–diffusion equations considered in this study are as follows:
\subsubsection{Burgess equation}
In \cite{ganji2021mathematical}, the fractional form of the one-dimensional Burgess equation was investigated, while \cite{nayied2023numerical} employed wavelet-based techniques to numerically simulate the integer-order Burgess equation. It is expressed as \cite{nayied2023numerical}

\begin{align}\label{burgis2}
\frac{\partial u(t,x)}{\partial t} &= \frac{1}{2}\frac{\partial^{2}u(t,x)}{\partial x^{2}} + R(t,x), \quad &&\text{on} \quad [0,T] \times  \mathbf{D_{1}},\\ 
u(0, x) &= u_{0}(x), \quad && \text{on} \quad  \mathbf{D_{1}}, \\  
    u &= \Gamma_{1}, \quad  u = \Gamma_{2}, \quad && \text{on} \quad  [0,T] \times \partial \mathbf{D_{1}}.  
\end{align}
Where \(\Gamma_{1}\) and \( \Gamma_{2}\) are known functions.
\subsubsection{Extended Fisher–Kolmogorov equation}  
The EFK equation represents a nonlinear biharmonic equation. It is expressed as \cite{ilati2020analysis}

\begin{align}\label{eq:eqmain}  
    u_{t} +  \gamma \Delta^{2}u - \Delta u + F(u)  &= 0,  && \text{on} \quad [0,T] \times  \mathbf{D_{2}}, \\  
    u(0,x) &= u_{0}(x),  &&  \text{on} \quad  \mathbf{D_{2}} , \\  
    u = \Gamma_{1}, \quad \Delta u &= \Gamma_{2}, &&  \text{on} \quad  [0,T] \times \partial \mathbf{D_{2}}.  
\end{align}  

Here, \( u_t \) represents the time evolution,  
\( \gamma \Delta^2 u \) accounts for higher-order diffusion (biharmonic diffusion),  
\( -\Delta u \) corresponds to standard diffusion,  and  
\(F(u)= u^3 - u \) is a nonlinear reaction term.  

In this context, \( \Gamma_{3} \) and \( \Gamma_{4} \) denote known functions, while \(t \in [0,T], x \in \mathbf{D_{2}} \subset \mathbb{R}^{d}\) and \(u \in [0,T] \times  \mathbf{D_{2}} \rightarrow \mathbb{R} \) represents a confined domain. The parameter \( \gamma \) is a strictly positive constant. An essential characteristic of the EFK equation is its energy dissipation law, defined through the energies dissipation law, defined through 
the energy functional can be written as \cite{pei2024unconditional}
\begin{equation}\label{eq:engergy2}
E(u) = \int\limits_{\mathbf{D_{2}}} \left( \frac{\gamma}{2} |\Delta u|^2 + \frac{1}{2} |\nabla u|^2 + \frac{1}{4} (1 - u^2)^2 \right) dx.
\end{equation}

\subsection{The Underlying Abstract PDE}  
Consider separable Banach spaces \( X \) and \( Y \) with norms \( \| \cdot \|_{X} \) and \( \| \cdot \|_{Y} \), respectively. To be precise, define \( Y = L^p(\mathbbm{D}; \mathbb{R}^m) \) and \( X = W^{s,q}(\mathbbm{D}; \mathbb{R}^m) \), where \( m \geq 1 \), \( 1 \leq p, q < \infty \), and \( s \geq 0 \), with \( \mathbbm{D} \subset {\mathbb{R}}^{\overline{d}} \) for some \( \bar{d} \geq 1 \).
 \( \mathbbm{D}= \mathbf{D}_T = [0, T] \times \mathbf{D} \subset \mathbb{R}^d \). Let \( X^{\ast} \subset X \) and \( Y^{\ast} \subset Y \) be closed subspaces equipped with norms \( \|\cdot\|_{X^{\ast}} \) and \( \|\cdot\|_{Y^{\ast}} \), respectively. The forward problem is well-posed, as all necessary information is available, while the inverse problem is inherently ill-posed due to missing or incomplete information.
\subsubsection{Forward problems}\label{sec:forward}
The abstract formulation of the governing PDE is  
\begin{equation}  
\label{eq:pde}  
\mathcal{D}(u) = \mathbf{f},  
\end{equation}  
where \( \mathcal{D} \) represents a differential operator mapping \( X^{\ast} \) to \( Y^{\ast} \), and \( \mathbf{f} \in Y^{\ast} \) satisfies the following conditions:  
\begin{equation}  
\label{eq:assm1}  
\begin{aligned}  
&\text{(H1)}: \quad \|\mathcal{D}(u)\|_{Y^{\ast}} < \infty, \quad \forall u \in X^{\ast} \text{ with } \|u\|_{X^{\ast}} < \infty. \\  
&\text{(H2)}: \quad \|\mathbf{f}\|_{Y^{\ast}} < \infty.  
\end{aligned}  
\end{equation}  
Additionally, assume that for each \( \mathbf{f} \in Y^{\ast} \), a unique solution \( u \in X^{\ast} \) exists for \eqref{eq:pde}, subject to approximate boundary and initial conditions given by  

\begin{equation}  
\mathcal{B}(u) = u_b \quad \text{on } \partial \mathbf{D}, \quad u(0, x) = u_0 \quad \text{on } \mathbf{D}.  
\end{equation}  
Here, \( \mathcal{B} \) represents a boundary operator, \( u_b \) is the prescribed boundary data, and \( u_0(x) \) denotes the initial condition.
\subsubsection{Inverse problems}
The problem is considered with unknown boundary and initial conditions, rendering the forward problem defined in \eqref{eq:pde} ill-posed. Let the solution  $u$  satisfy the given equation within the subdomain $\mathbf{D}'_T$. The operator $\mathcal{L}$ applied to $u $ in this region is given by a prescribed function  $g$  expressed as:

\begin{equation}  
\mathcal{L}(u) = g, \quad \forall (t, \boldsymbol{x}) \in \mathbf{D}'_T.  
\end{equation}  

Here, \(\mathbf{D}'\) is a subset of \(\mathbf{D}\), and the spatiotemporal domain is defined as
\(
\mathbf{D}'_{T} = [0, T] \times \mathbf{D}' \subset [0, T] \times \mathbf{D}.
\)\\

In particular, for the Burgess equation we take \(\mathbf{D} = \mathbf{D}_1\), so that
\(
[0, T] \times \mathbf{D}_1' \;\subset\; [0, T] \times \mathbf{D},
\)
while for the EFK equation we set \(\mathbf{D} = \mathbf{D}_2\), giving
\(
[0, T] \times \mathbf{D}_2' \;\subset\; [0, T] \times \mathbf{D}.
\)

\subsection{Quadrature Rules}
Let \( \mathbbm{D} \) represent a domain and \( h \) be an integrable function defined by \( h: \mathbbm{D} \to \mathbb{R} \). Consider the space-time domain \( \mathbbm{D}= \mathbf{D}_T = [0, T] \times \mathbf{D} \subset \mathbb{R}^d \), where \( \bar{d} = 2d + 1 \geq 1 \). The function \( h \) is given on \(  \mathbbm{D} \) as follows:   
\begin{equation}  
h = \int\limits_{ \mathbbm{D} } h(z) \, dz,
\end{equation}  
where \( dz \) denotes the \( \bar{d} \)-dimensional Lebesgue measure. For quadrature, we select points \( z_i \in  \mathbbm{D} \) for \( 1 \leq i \leq N \), along with corresponding weights \( w_i \). The quadrature approximation then takes the form:  
\begin{equation}  
h_N = \sum_{i=1}^{N} w_i h(z_i).  
\end{equation}  
Here, \( z_i \) are the quadrature points. For moderately high-dimensional problems, low-discrepancy sequences such as Sobol and Halton sequences can be employed as quadrature points. For very high-dimensional problems (\( d \gg 20 \)), Monte Carlo quadrature becomes the preferred method for numerical integration \cite{caflisch1998monte}, where the quadrature points are selected randomly and independently.  \\
For a set of weights \( w_i \) and quadrature points \( y_i \), we assume that the associated quadrature error adheres to the following bound:
\begin{equation}
    \label{eq:assm3}
    \left|\overline{h} - \overline{h}_N\right| \leq C_{quad}
    \left(\|h\|_{Z^{\ast}},\bar{d} \right) N^{-\alpha},
\end{equation}
for some $\alpha > 0$. 
\subsection{Training Points}\label{Training}
Physics-informed neural networks require four types of training points as described in \cite{mishra2021physics, mishra2022estimates}: interior points \( \textit{S}_{\text{int}}\), temporal boundary points \(\textit{S}_{\text{tb}}\), spatial boundary points \(\textit{S}_{\text{sb}}\), and data points \(\textit{S}_{\boldsymbol{d}}\). Figs.\ref{figt} and \ref{figtt} illustrate the training points used in forward and inverse problems. The training set for forward problems is given by  
\[
\boldsymbol{S} = \textit{S}_{\text{int}} \cup \textit{S}_{\text{sb}} \cup  \textit{S}_{\text{tb}}.
\]  
For the inverse problem, additional training points are required, ie, data training points \( \textit{S}_{\boldsymbol{d}} \). The defined training points \( \textit{S}_{\text{int/sb/tb}/\boldsymbol{d} }\) correspond to quadrature points with weights \( w_{j}^{\text{int}/\text{sb}/\text{tb}/\boldsymbol{d}} \), determined by an appropriate quadrature rule.  
In domains \( \mathbf{D} \) that are logically rectangular, the training set can be constructed using Sobol points or randomly selected points. Thus, we can define these training points \(N\) as follows:

\begin{figure}[htbp]
\centering
\includegraphics[height=0.4\textheight]{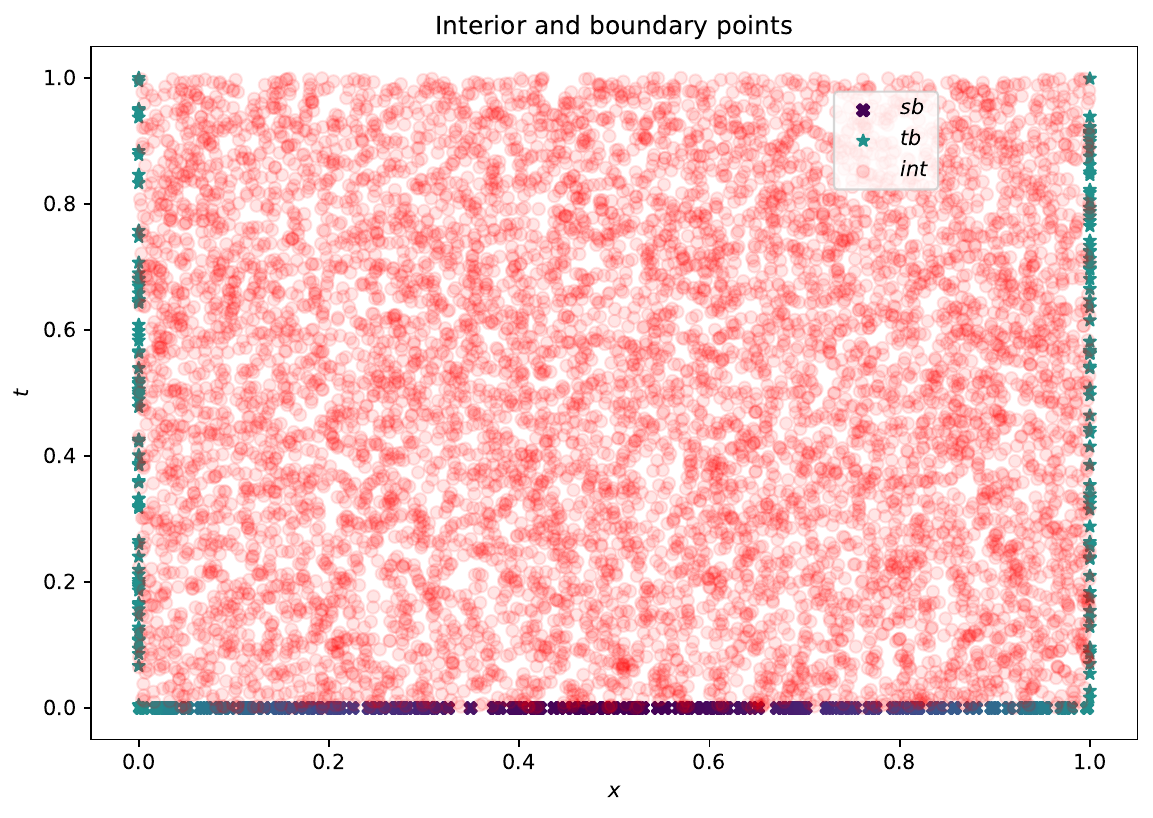}
\caption{\textbf{Training points (forward problem):} The training set \( \boldsymbol{S} \) consists of randomly chosen points. Red dots denote interior points, whereas green and blue dots correspond to temporal and spatial boundary points.
}
\label{figt}
\end{figure} 

\subsubsection{Interior training points}
The interior training points are denoted by $\textit{S}_{\text{int}} = \left\lbrace z_{j}^{\text{int}} \right\rbrace $ for $ 1 \leq j \leq N_{\text{int}} $, where $ z_{j}^{\text{int}} = \left( t_{j}^{\text{int}}, x_{j}^{\text{int}}\right) $. Here, $t_{j}^{\text{int}} \in \left[ 0, T \right] $, $x_{j}^{\text{int}} \in  \mathbf{D} $ for all \( j \).

\subsubsection{Temporal boundary training points}
The temporal boundary points are represented as \( \textit{S}_{\text{tb}} = \left\lbrace z_{j}^{\text{tb}} \right\rbrace \), for \( 1 \leq j \leq N_{\text{tb}} \), with \( z_{j}^{\text{tb}} = \left( x_{j}^{\text{tb}} \right) \). Here, \( x_{j}^{\text{tb}} \in \mathbf{D} \), \(\forall\) \( j \).

\subsubsection{Spatial boundary training points}
The spatial boundary points are denoted as \( \textit{S}_{\text{sb}} = \left\lbrace z_{j}^{\text{sb}} \right\rbrace \), for \( 1 \leq j \leq N_{\text{sb}} \), where \( z_{j}^{\text{sb}} = \left( t_{j}^{\text{tb}}, x_{j}^{\text{tb}} \right) \). In this case, \( t_{j}^{\text{tb}} \in \left[ 0, T \right] \), \( x_{j}^{\text{tb}} \in \partial \mathbf{D} \).
\subsubsection{Data training points}
The data training set is defined as \( \textit{S}_{\boldsymbol{d}} = \left\lbrace y_{j}^{\boldsymbol{d}} \right\rbrace \) for \( 1 \leq j \leq N_{\boldsymbol{d}} \), where \( y_{j}^{\boldsymbol{d}} \in  \mathbf{D}_{T}' \).\\

\begin{figure}[htbp]
\centering
\includegraphics[height=0.3\textheight]{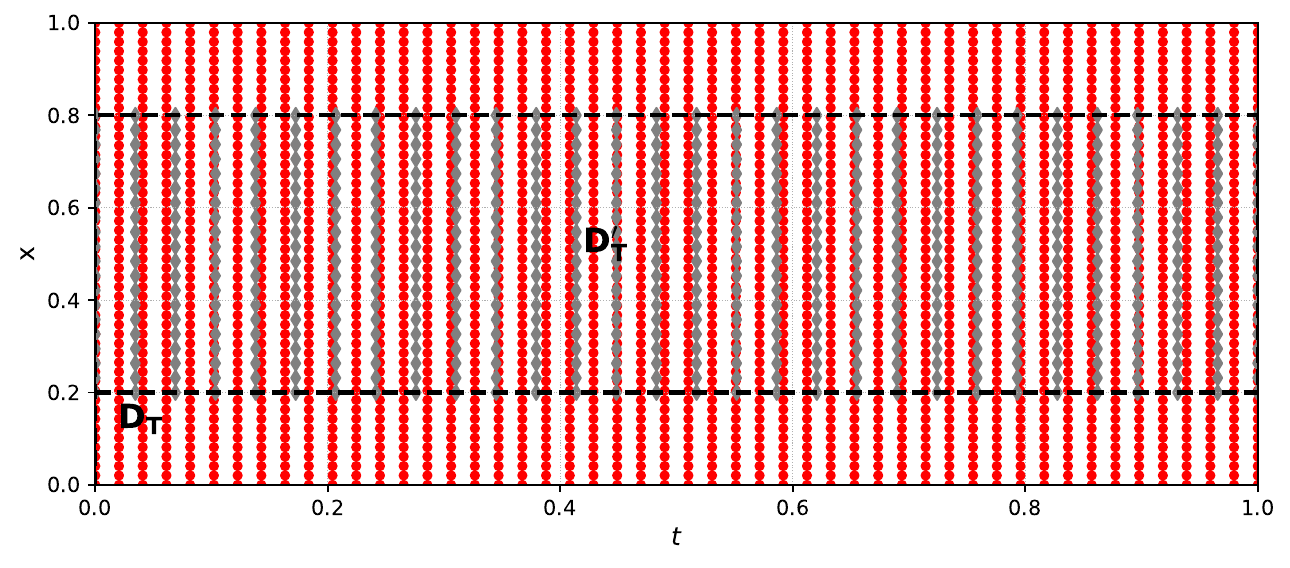}
\caption{Training points (inverse problem): A visualization of the training set \( \boldsymbol{S} \) with randomly sampled training points. Red dots denote interior points, while gray dots correspond to Sobol points.
}
\label{figtt}
\end{figure}
\begin{figure}[htbp]
\centering
\includegraphics[height=0.35\textheight]{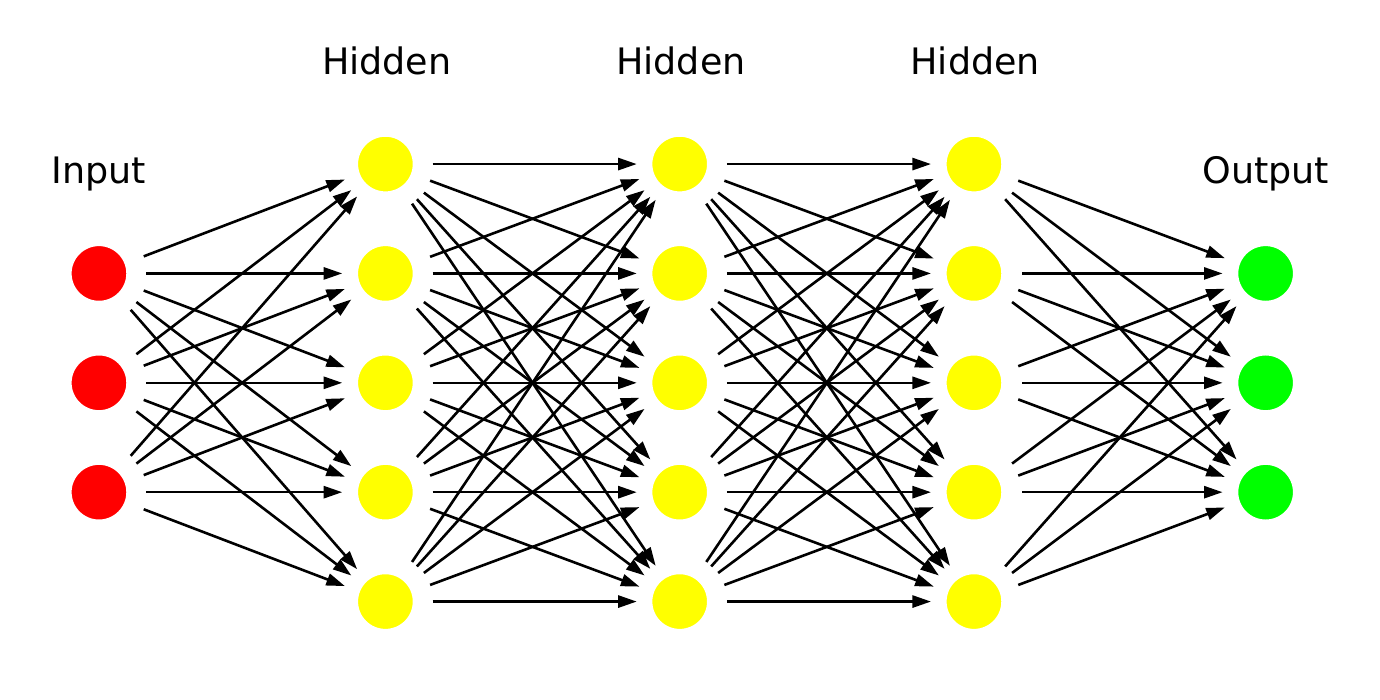}
\caption{In this diagram, neurons in the input layer are shown in red, those in the hidden layer are represented in yellow, and neurons in the output layer are depicted in green.
}
\label{figttt}
\end{figure}

\subsection{Neural Networks}
The PINN operates as a feed-forward neural network, as depicted in Fig.~\ref{figttt}. Without an activation function, a neural network functions similarly to a multiple regression model. The activation function introduces non-linearity, enabling the network to learn and perform complex tasks. Examples of activation functions include sigmoid, hyperbolic tangent (tanh), and ReLU \cite{goodfellow2016deep}. 
The network receives an input \( y = (t, x) \in \mathbbm{D} \), and can be formulated as an affine transformation:
\begin{equation}\label{eqn:N}
u_{\theta}(y) = C_K \circ  \sigma \circ C_{K-1} \circ \dots \circ \dots \sigma \circ C_1 (y).
\end{equation}
Here, \( \circ \) denotes function composition, and \( \sigma \) represents activation functions.
For each layer \( k \) where \( 1 \leq k \leq K \), the transformation is given by:
\begin{equation}
C_k z_k = W_k z_k + b_k, \quad \text{where} \quad W_k \in \mathbb{R}^{d_{k+1} \times d_k}, \, z_k \in \mathbb{R}^{d_k}, \quad \text{and} \quad b_k \in \mathbb{R}^{d_{k+1}}.
\end{equation}
To maintain consistency, we define \( d_1 = \bar{d} = 2d + 1 \), where \( d \) is the spatial dimension, and set \( d_K = 1 \) for the output layer. Structurally, the network consists of an input layer, an output layer, and \( K-1 \) hidden layers, subject to the condition \( 1 < K < \mathbb{N} \).

Each hidden layer \( k \), comprising \( d_k \) neurons, processes an input vector \( z_k \in \mathbb{R}^{d_k} \). The transformation begins with the linear mapping \( C_k \), followed by the application of the activation function \( \sigma \). The total number of neurons in the network is given by \( 2d + 2 + \sum_{k=2}^{K-1} d_k \).

The set of network parameters, including weights and biases, is denoted as \( \theta = \left\lbrace W_k, b_k \right\rbrace \). Additionally, the weight parameters alone are represented as \( \theta_w = \left\lbrace W_k \right\rbrace \) for all \( 1 \leq k \leq K \) \cite{mishra2021physics, mishra2022estimates}. The parameters \( \theta \) belong to the space \( \theta' \subset \mathbb{R}^P \), where \( P \) represents the total number of parameters:
\begin{equation}
P = \sum_{k=1}^{K-1} \left( d_k + 1 \right) d_{k+1}.
\end{equation}

\subsection{Residuals}
This section describes the residuals linked to different training sets, including interior, temporal, spatial and data points used for inverse problems. The primary objective is to minimize these residuals. Optimization will incorporate stochastic gradient descent techniques, such as ADAM for first-order optimization, along with higher-order methods like variants of the BFGS algorithm. The PINN \( u_{\theta} \) depends on tuning parameters \( \theta \in \theta' \), which correspond to the network's weights and biases. In a standard deep learning framework, training involves adjusting these parameters to ensure that the neural network approximation \( u_{\theta} \) closely matches the exact solution \( u \). 
The interior residual is defined as:  
 
\begin{equation}\label{eq:main_res}
\begin{aligned}
\mathfrak{R}_{\text{int},\theta} = \mathfrak{R}_{\text{int},\theta}(t,x), \quad \forall (t,x) \in [0,T] \times \mathbf{D}.
\end{aligned}
\end{equation}  
It can be expressed in terms of the differential operator as follows:  
\begin{equation}\label{eqn:Ra}
\begin{aligned}
\mathfrak{R}_{\text{int}, \theta} =  \mathcal{D}(u_{\theta}) - \mathbf{f}.
\end{aligned}
\end{equation}  

Residuals corresponding to initial, boundary, and data points are formulated as:

\begin{equation}\label{eqn:Rab}
\begin{aligned}
\mathfrak{R}_{\text{tb}} &= \mathfrak{R}_{\text{tb}, \theta} = u_{\theta} - u_{0}, \quad \forall x \in \mathbf{D}, \\
\mathfrak{R}_{\text{sb}} &= \mathfrak{R}_{\text{sb}, \theta} = u_{\theta} - u_{b}, \quad \forall (t, x) \in \partial{\mathbf{D}}.
\end{aligned}
\end{equation}
Additionally, the residual for data points is given by:
\begin{equation}\label{eqn:Residual_data}
\mathfrak{R}_{ {\boldsymbol{d}}} = \mathcal{L}(u_{\theta}) - g, \quad \forall (t, x) \in \mathbf{D}'_{T}.
\end{equation}
\subsubsection*{Residual-Based Weighting}
In this section, the PDE residual is modulated by a weighting function to adaptively emphasize or de-emphasize specific collocation points during training. 
We consider two distinct weighting strategies, denoted as \textbf{RWa} and \textbf{RWb}, defined as follows.

\subsubsection*{RWa}
Mathematically, we define
\[
\mathrm{W_{RWa}} = \sigma\left(-\lambda_F \, \mathrm{detach}(\mathfrak{R}_{\mathrm{int}, \theta})\right),
\]
where
\[
\sigma(z) = \frac{1}{1 + e^{-z}}
\]
is the sigmoid activation function, \(0 < \lambda_F \leq 1\) is a scaling parameter, and 
$\mathrm{detach}(\mathfrak{R}_{\mathrm{int}, \theta})$ indicates that the interior residual $\mathfrak{R}_{\mathrm{int}, \theta}$ is treated as a constant during backpropagation to prevent gradient flow through the weighting term.

\subsubsection*{RWb}
Mathematically, we define
\[
\mathrm{W_{RWb}} = \tanh\!\left( \mathrm{softplus}\left(-\lambda_H \, \mathrm{detach}(\mathfrak{R}_{\mathrm{int}, \theta})\right) \right),
\]
where
\[
\mathrm{softplus}(z) = \log\!\left( 1 + e^{z} \right)
\]
is the softplus activation function, \(0 < \lambda_H \leq 1\) is a scaling parameter, and 
$\mathrm{detach}(\mathfrak{R}_{\mathrm{int}, \theta})$ indicates that the interior residual is treated as a constant during backpropagation.
After applying the scaling through the weight functions $\mathrm{W_{RWa}}$ or $\mathrm{W_{RWb}}$, 
the original PDE residual $\mathfrak{R}_{\mathrm{int},\theta}$ is transformed into the weighted residual
\[
\mathfrak{R}'_{\mathrm{int},\theta}
= \mathrm{W} \cdot \mathfrak{R}_{\mathrm{int},\theta},
\]
where $\mathrm{W} \in \{\mathrm{W_{RWa}}, \mathrm{W_{RWb}}\}$ denotes the chosen weight function.

The goal is to determine the optimal tuning parameters \( \theta \in \theta' \) that minimize the residual in the forward problem:
\begin{equation}\label{eqn:R}
  \theta^{\ast} \in \theta' : \theta^{\ast} = \arg \min_{\theta \in \theta'} \left( \Vert \mathfrak{R'}_{\text{int}, \theta} \Vert^{2}_{L^{2}(\mathbf{D}_{T})} + \Vert \mathfrak{R}_{\text{sb}, \theta} \Vert^{2}_{L^{2}(\mathbf{[0,T] \times\partial D})} + \Vert \mathfrak{R}_{\text{tb}, \theta} \Vert^{2}_{L^{2}(\mathbf{D})} \right).
\end{equation}
For the inverse problem, an additional term corresponding to the data residual \( \mathrm{R}_{\boldsymbol{d}} \) is introduced in Eq.~\eqref{eqn:R}, leading to the following minimization problem:
\begin{equation}\label{eqn:Rb}
  \theta^{\ast} \in \theta' : \theta^{\ast} = \arg \min_{\theta \in \theta'} 
  \left( 
  \Vert \mathfrak{R'}_{\text{int}, \theta} \Vert^{2}_{L^{2}( \mathbf{D}_{T})} 
  + \Vert \mathfrak{R}_{\text{sb}, \theta} \Vert^{2}_{L^{2}( \mathbf{[0,T] \times\partial D})}  + \Vert \mathfrak{R}_{\boldsymbol{d}, \theta} \Vert^{2}_{L^{2}( \mathbf{D}'_{T})}  +\Vert \mathfrak{R}_{\text{tb}, \theta} \Vert^{2}_{L^{2}(\mathbf{D})}
  \right).
\end{equation}
Since the integrals in Eqs.~\eqref{eqn:R} and \eqref{eqn:Rb} involve the \( L^{2} \) norm, an exact computation is not feasible. Instead, numerical quadrature methods are employed for approximation.

\subsection{Loss Functions and Optimization}\label{loss}
The integrals \eqref{eqn:R} is approximated using the following loss functions for forward problems:
\begin{equation}
\mathscr{L}_{1}(\theta)=\sum\limits_{j=1}^{N_{sb}}w_{j}^{sb}\vert  \mathfrak{R}_{\text{sb}, \theta} (z_{j}^{sb}) \vert^{2}+ \sum\limits_{j=1}^{N_{tb}}w_{j}^{tb}\vert  \mathfrak{R}_{\text{tb}, \theta} (z_{j}^{tb}) \vert^{2}+\lambda \sum\limits_{j=1}^{N_{int}}w_{j}^{int}\vert  \mathfrak{R'}_{\text{int}, \theta} (z_{j}^{int}) \vert^{2} \label{eqn:La},
\end{equation}
where $\lambda$ is residual parameter.
The integrals \eqref{eqn:Rb} is approximated using the following loss functions for inverse problems:
\begin{equation}
\mathscr{L}_{2}(\theta)=\sum\limits_{j=1}^{N_{d}}w_{j}^{d}\vert  \mathfrak{R}_{\boldsymbol{d}, \theta} (z_{j}^{d}) \vert^{2}+\sum\limits_{j=1}^{N_{sb}}w_{j}^{sb}\vert  \mathfrak{R}_{\text{sb}, \theta} (z_{j}^{sb}) \vert^{2}+ \sum\limits_{j=1}^{N_{tb}}w_{j}^{tb}\vert  \mathfrak{R}_{\text{tb}, \theta} (z_{j}^{tb}) \vert^{2} + \lambda \sum\limits_{j=1}^{N_{int}}w_{j}^{int}\vert  \mathfrak{R'}_{\text{int}, \theta} (z_{j}^{int}) \vert^{2}  \label{eqn:Lb},
\end{equation}
where $\lambda$ is residual parameter.
The loss function minimization is regularized as follows:
\begin{equation}\label{eqn:L}
\theta^{\ast} = \arg \min_{\theta \in \theta'}(\mathscr{L}_{i}(\theta)+ \lambda_{reg}\mathscr{L}_{reg}(\theta)),
\end{equation}
where \( i = 1, 2 \). In deep learning, regularization helps prevent over-fitting. A common form of regularization is \( \mathscr{L}_{\text{reg}}(\theta) = \Vert \theta \Vert_q^q \), where \( q = 1 \) (for \( L^{1} \) regularization) or \( q = 2 \) (for \( L^{2} \) regularization).
The regularization parameter \( \lambda_{\text{reg}} \) balances the trade-off between the loss function \( \mathscr{L} \) and the regularization term, where \( 0 \leq \lambda_{\text{reg}} \ll 1 \).
Stochastic gradient descent algorithms such as ADAM will be used for optimization, as they are widely adopted for first-order methods. Additionally, second-order optimization strategies, including different versions of the BFGS algorithm, may be employed. The objective is to determine the optimal neural network solution \( u^{\ast} = u_{\theta^{\ast}} \) using the training dataset. The process begins with an initial parameter set \( \bar{\theta} \in \theta' \), and the corresponding network output \( u_{\bar{\theta}} \), residuals, loss function, and gradients are computed iteratively. Ultimately, the optimal solution, denoted as \( u^{\ast} = u_{\theta^{\ast}} \), is obtained through PINN. The local minimum in Eq.~\eqref{eqn:L} is approximated as \( \theta^{\ast} \), yielding the deep neural network solution \( u^{\ast} = u_{\theta^{\ast}} \), which serves as an approximation to \( u \) in low grades tumors models.

The hyper-parameters used in numerical experiments are summarized in Table~\ref{table_1}. The PINN framework for solving the reaction diffusion equations follows the methodologies outlined in \cite{mishra2021physics, mishra2022estimates, bai2021physics, mishra2023estimates}. The illustration in Fig.~\ref{figttt} represents the PINN framework.
 Below, Algorithm \ref{alg1} is presented for forward problems, while Algorithm \ref{alg2} addresses inverse problems (RW-PINN):
\begin{table}[!ht]
\centering
\caption{The configurations of hyper-parameters and the frequency of retraining utilized in ensemble training for PINN.}

\begin{tabular}{||c c c c c c||}
\hline
Experiments & $K-1$ & $\bar{d}$ & $\lambda$ & $\lambda_{\text{reg}}$ & $n_{\theta}$ \\ [0.5ex] 
 \hline
\ref{Example2} & 4 & 20 & 0.1, 1, 10 & 0 & 10\\ 
\hline
\ref{Example3} & 4 & 20 & 0.1, 1, 10 & 0 & 10,4 \\ 
\hline
\ref{Example3+1} & 4 & 20 & 0.1, 1, 10 & 0 & 12 \\ 
\hline
\ref{Example4} & 4 & 28 & 0.1, 1, 10 & 0 & 10,4 \\ 
\hline
\ref{Example5} & 4 & 20 & 0.1, 1, 10 & 0 & 10 \\ 
\hline
\ref{Example6} & 4 & 20 & 0.1, 1, 10 & 0 & 10 \\ 
\hline
\ref{Example7} & 4 & 20, 36, 42 & 0.1, 1, 10 & 0 & 10,10,4 \\ 

\hline
\end{tabular}
\label{table_1}
\end{table}

\begin{algorithm} 
\label{alg1} {\bf The RW-PINN framework is employed for estimating reaction diffusion equations in forward problems}
\begin{itemize}
\item [{\bf Inputs}:] Define the computational domain, problem data, and coefficients for the reaction diffusion equations. Specify quadrature points and weights for numerical integration. Choose a non-convex gradient-based optimization method for training the neural network.
\item [{\bf Aim}:] Develop a RW-PINN approximation \( u^{\ast} = u_{\theta^{\ast}} \) for solving the model.

\item [{\bf Step $1$}:]  Select the training points following the methodology described in Section~\ref{Training}.
\item [{\bf Step $2$}:] Initialize the network with parameters \( \bar{\theta} \in \theta' \) and compute the following: neural network output \( u_{\bar{\theta}} \) Eq.~(\ref{eqn:N}), PDE residual Eq.~(\ref{eqn:Ra}), boundary residuals Eq.~(\ref{eqn:Rab}), loss function Eq.~(\ref{eqn:La}), Eq.~(\ref{eqn:L}), and gradients required for optimization.
\item [{\bf Step $3$}:] Apply the optimization algorithm iteratively until an approximate local minimum \( \theta^{\ast} \) of Eq.~(\ref{eqn:L}) is obtained. The trained network \( u^{\ast} = u_{\theta^{\ast}} \) serves as the RW-PINN solution for the reaction diffusion equations.
\end{itemize}
\end{algorithm}

\begin{figure}[htbp]
\centering
\includegraphics[height=0.4\textheight]{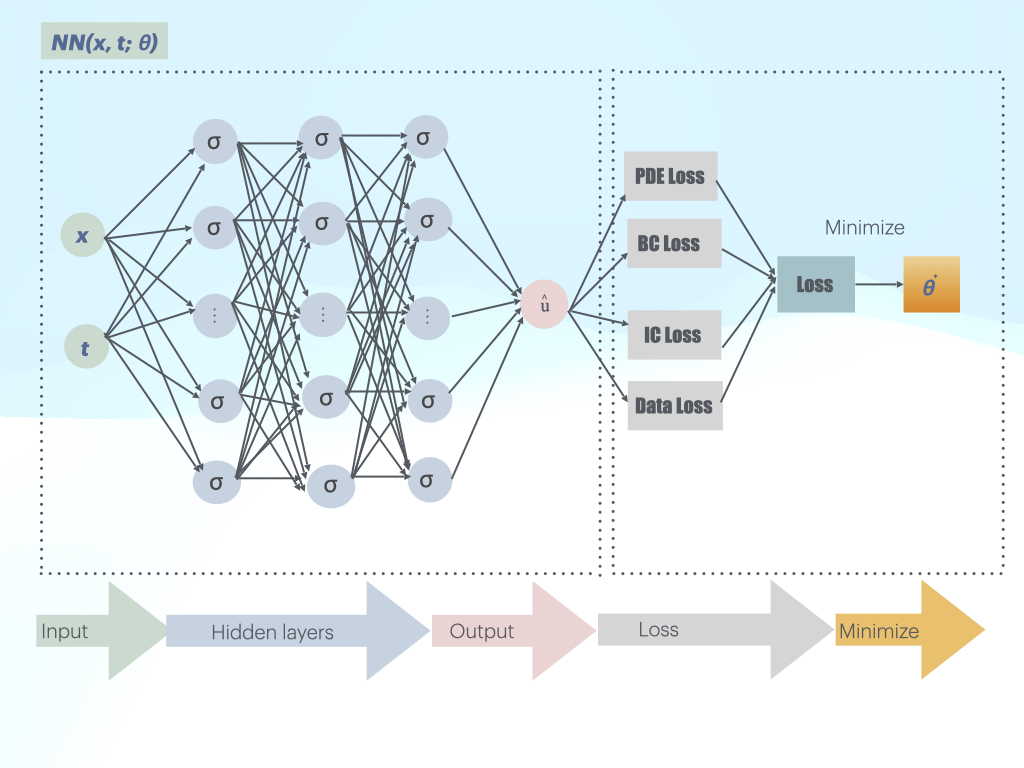}
\caption{Schematic representation of the PINN framework.}
\label{figttt}
\end{figure}

\begin{algorithm} 
\label{alg2} {\bf The RW-PINN framework is employed for estimating reaction diffusion equations in  inverse problems}
\begin{itemize}
\item [{\bf Inputs}:] Define the computational domain, problem data, and coefficients for the reaction diffusion equations. Specify quadrature points and weights for numerical integration. Choose a suitable non-convex gradient-based optimization method.
\item [{\bf Aim}:]  Construct a PINN approximation \( u^{\ast} = u_{\theta^{\ast}} \) to estimate the solution \( u \) of reaction diffusion equations for inverse problems.
\item [{\bf Step $1$}:] Select training points according to the methodology outlined in Section~\ref{Training}.
\item [{\bf Step $2$}:] Initialize the neural network with parameters \( \bar{\theta} \in \theta' \) and compute the following components: neural network output \( u_{\bar{\theta}} \) Eq.~(\ref{eqn:N}), PDE residual Eq.~(\ref{eqn:Ra}), data residuals Eq.~(\ref{eqn:Residual_data}), loss function Eq.~(\ref{eqn:Lb}), Eq.~(\ref{eqn:L}), and gradients for optimization.
\item [{\bf Step $3$}:] Apply the optimization algorithm iteratively until an approximate local minimum \( \theta^{\ast} \) of Eq.~(\ref{eqn:L}) is reached. The trained network \( u^{\ast} = u_{\theta^{\ast}} \) serves as the RW-PINN solution for the the reaction diffusion equations.
\end{itemize}
\end{algorithm}

\subsection{Estimation on Generalization Error} 
Let the spatial domain be \(\mathbf{D} = [0, 1]^{d}\), where \(d\) denotes the spatial dimension. This section focuses on obtaining an accurate estimate of the generalization error, also known as the total error, for the trained neural network \(u^{\ast} = u_{\theta^{\ast}}\). This result arises from the application of the RW-PINN algorithms \ref{alg1} and \ref{alg2}. The error can be expressed as follows:
\begin{equation}
    \label{eq:gen}
    \boldsymbol{\mathcal{E}}_{G}:= \left(\int\limits_0^T \int\limits_0^1 |u(t,x) - u^{\ast}(t,x)|^2 dx dt \right)^{\frac{1}{2}}.
\end{equation}
This approach is outlined in \cite{mishra2021physics}, \cite{bai2021physics}, \cite{mishra2022estimates} and \cite{mishra2023estimates}. This section provides an estimation of the generalization error, as defined in equation \eqref{eq:gen}, based on the training error.
For the abstract PDE equation \eqref{eq:pde}, the generalization error is analyzed by expressing it in terms of the training error, which is defined as follows:
\begin{align}
    (\boldsymbol{\mathcal{E}}_{T}^{int})^2 &= \sum\limits_{n=1}^{N_{int}} w^{int}_n|\mathfrak{R'}_{int,\theta^{\ast}}|^2, \quad
    (\boldsymbol{\mathcal{E}}_T^{sb})^2 = \sum\limits_{n=1}^{N_{sb}} w^{sb}_n| \mathfrak{R}_{sb,\theta^{\ast}}|^2, \notag \\
    (\boldsymbol{\mathcal{E}}_T^{tb})^2 &= \sum\limits_{n=1}^{N_{tb}} w^{tb}_n| \mathfrak{R}_{tb,\theta^{\ast}}|^2, \quad
    (\boldsymbol{\mathcal{E}}_T^{d})^2 = \sum\limits_{n=1}^{N_{\boldsymbol{d}}} w^{\boldsymbol{d}}_n| \mathfrak{R}_{\boldsymbol{d},\theta^{\ast}}|^2.
    \label{eq:training_errors}
\end{align}
For the EFK equation, we can modify \(\boldsymbol{\mathcal{E}}_T^{sb}\) as in \cite{bai2021physics}:
\begin{equation}
    (\boldsymbol{\mathcal{E}}_T^{sb})^2 =\sum\limits_{n=1}^{N_{sb}} \sum\limits_{i=1}^{4} w^{sb}_n| \mathfrak{R}_{sbi,\theta^{\ast}}|^2.
\end{equation}
The training error can be directly computed \emph{a posteriori} using the loss function equation \eqref{eqn:L}. Additionally, the following assumptions on the quadrature error are required, similar to equations \eqref{eqn:L} and \eqref{eqn:Lb}. For any function \( h \in C^k(\mathbf{D}) \), the quadrature rule, defined using quadrature weights \( w^{tb}_n \) at points \( x_n \in \train_{tb} \) for \( 1 \leq n \leq N_{tb} \), satisfies the bound  
\begin{equation}
    \label{eq:hquad1}
    \left| \int\limits_{\mathbf{D}} h(x) dx - \sum\limits_{n=1}^{N_{tb}} w^{tb}_n h(x_n) \right| \leq C^{tb}_{quad}(\|h\|_{C^k}) N_{tb}^{-\alpha_{tb}}.
\end{equation}
For any function $g \in C^k(\partial \mathbf{D} \times [0,T])$, the quadrature rule corresponding to quadrature weights $w^{sb}_n$ at points $(x_n,t_n) \in \train_{sb}$, with $1 \leq n \leq N_{sb}$, satisfies 
\begin{equation}
    \label{eq:hquad2}
    \left| \int\limits_0^T \int\limits_{\partial \mathbf{D}} h(x,t) ds(x) dt - \sum\limits_{n=1}^{N_{sb}} w^{sb}_n h(x_n,t_n)\right| \leq C^{sb}_{quad}(\|h\|_{C^k}) N_{sb}^{-\alpha_{sb}}.
\end{equation}
Finally, for any function $h \in C^\ell(\mathbf{D} \times [0,T])$, the quadrature rule corresponding to quadrature weights $w^{int}_n$ at points $(x_n,t_n) \in \train_{int}$, with $1 \leq n \leq N_{int}$, satisfies 
\begin{equation}
    \label{eq:hquad3}
    \left| \int\limits_0^T \int\limits_{\mathbf{D}} h(x,t) dx dt - \sum\limits_{n=1}^{N_{int}} w^{int}_n h(x_n,t_n)\right| \leq C^{int}_{quad}(\|h\|_{C^\ell}) N_{int}^{-\alpha_{int}}.
\end{equation}
In the above, $\alpha_{int},\alpha_{sb},\alpha_{tb} > 0$ and in principle, different-order quadrature rules can be used.
The generalization error for the Burgess equation and the EFK equation, obtained using Algorithm~\ref{alg1}, is given in the following form:
\begin{equation}
    \label{eq:hegenb}
    \boldsymbol{\mathcal{E}}_{G} \leq C_1 \left(\boldsymbol{\mathcal{E}}_T^{tb}+\boldsymbol{\mathcal{E}}_T^{int}+C_2(\boldsymbol{\mathcal{E}}_T^{sb})^{\frac{1}{2}} + (C_{quad}^{tb})^{\frac{1}{2}}N_{tb}^{-\frac{\alpha_{tb}}{2}} +  (C_{quad}^{int})^{\frac{1}{2}}N_{int}^{-\frac{\alpha_{int}}{2}} + C_2  (C_{quad}^{sb})^{\frac{1}{4}}N_{sb}^{-\frac{\alpha_{sb}}{4}} \right),
\end{equation}
where the constants \(C_1\) and \(C_2\) are shown in \ref{thm:bergess} and \ref{thm:1}.

\subsection{Stability and Convergence of Multilayer Neural Network}
This section presents the stability and convergence analysis of the neural network for both models. For convenience, let \( u_{\theta} = U \). 
\subsubsection{Stability of multilayer neural network}
Here, \( L^{\infty} \) bounds are derived for both models.

\begin{theorem}\label{thm:bound}  
Let \( U \) be a neural network solution to the equation  
\begin{equation}\label{eq:Burcon}  
    \frac{\partial U}{\partial t} = \frac{1}{2} \frac{\partial^2 U}{\partial x^2} + R(U),  
\end{equation}  
where the reaction term \( R(U) \) satisfies the Lipchitz condition \ref{appendix:buress_lip} along with one of the following conditions:  

\begin{itemize}  
    \item \textbf{(i) Linear growth condition:}  
    If there exists a constant \( C > 0 \) such that  
    \begin{equation}  
        |R(U)| \leq C(1 + |U|),  
    \end{equation}  
    then \( U \) is uniformly bounded in \( L^2(\mathbf{D_1}) \), i.e., there exists a constant \( M > 0 \) such that  
    \begin{equation}  
        \sup_{t \in [0,T]} \|U(t)\|_{L^2(\mathbf{D})} \leq M.  
    \end{equation}  

    \item \textbf{(ii) Exponential decay condition }  
    \begin{equation}
    |R(U)| \leq C e^{-\alpha U},
\end{equation}
for some constants \( C, \alpha > 0 \). Then \( U \) is uniformly bounded in \( L^\infty(\mathbf{D_1}) \) and satisfies the estimate:
\begin{equation}
    \sup_{t \in [0,T]} \|U(t)\|_{L^\infty(\mathbf{D_1})} \leq C' e^{-\beta t} \|U_0\|_{L^2(\mathbf{D_1})},
\end{equation}
for some constant \( C' > 0 \), decay rate \( \beta > 0 \) and initial condition \(U_{0}\).
\end{itemize}  

\end{theorem}  

\begin{proof}
Multiplying the equation (\ref{eq:Burcon})  by \( U \) and integrating over \( \mathbf{D_1} \):  
\begin{equation}
    \int_{\mathbf{D_1}} U \frac{\partial U}{\partial t} \,dx = \frac{1}{2} \int_{\mathbf{D_1}} U \frac{\partial^2 U}{\partial x^2} \,dx + \int_{\mathbf{D_1}} U R(U) \,dx.
\end{equation}
Using integration by part and the Dirichlet boundary condition \( U = 0 \) on \( \partial \mathbf{D_1} \),  
\begin{equation}
    \int_{\mathbf{D_1}} U \frac{\partial^2 U}{\partial x^2} \,dx = - \int_{\mathbf{D_1}} \left( \frac{\partial U}{\partial x} \right)^2 \,dx \leq 0.
\end{equation}
Thus,  
\begin{equation}
    \int_{\mathbf{D_1}} U \frac{\partial U}{\partial t} \,dx \leq \int_{\mathbf{D_1}} U R(U) \,dx.
\end{equation}
Applying the Linear Growth Condition,  
\begin{equation}
    \int_{\mathbf{D_1}} U R(U) \,dx \leq C \int_{\mathbf{D_1}}(1 + |U|^2) \,dx.
\end{equation}
Using Gronwall’s inequality,  
\begin{equation}
    \sup_{t \in [0,T]} \|U(t)\|_{L^2(\mathbf{D_1})} \leq M.
\end{equation}
Multiply equation (\ref{eq:Burcon}) by \( |U|^{p-2} U \) and integrate over \( \mathbf{D_1} \):
\begin{equation}
    \int_{\mathbf{D_1}} |U|^{p-2} U \frac{\partial U}{\partial t} \,dx = \frac{1}{2} \int_{\mathbf{D_1}} |U|^{p-2} U \frac{\partial^2 U}{\partial x^2} \,dx + \int_{\mathbf{D_1}} |U|^{p-2} U R(U) \,dx.
\end{equation}
Using integration by parts and the Dirichlet boundary condition:
\begin{equation}
    \int_{\mathbf{D_1}} |U|^{p-2} U \frac{\partial^2 U}{\partial x^2} \,dx = - (p-1) \int_{\mathbf{D_1}} |U|^{p-2} \left( \frac{\partial U}{\partial x} \right)^2 \,dx \leq 0.
\end{equation}
Thus,
\begin{equation}
    \int_{\mathbf{D_1}} |U|^{p-2} U \frac{\partial U}{\partial t} \,dx \leq \int_{\mathbf{D_1}} |U|^{p-2} U R(U) \,dx.
\end{equation}
Applying the exponential decay condition:
\begin{equation}
    \int_{\mathbf{D_1}} |U|^{p-2} U R(U) \,dx \leq C \int_{\mathbf{D_1}} |U|^{p-2} U e^{-\alpha U} \,dx.
\end{equation}
For large \( |U| \), the term
\begin{equation}
    |U|^{p-1} e^{-\alpha U} \leq C e^{-\alpha U /2}.
\end{equation}
Since \( e^{-\alpha U /2} \) decays exponentially and dominates any polynomial growth, the integral remains bounded. Using Moser’s iteration,
\begin{equation}
    \sup_{t \in [0,T]} \|U(t)\|_{L^\infty(\mathbf{D_1})} \leq C' e^{-\beta t} \|U_0\|_{L^2(\mathbf{D_1})}.
\end{equation}
This proves that \( U \) remains strictly bounded in \( L^\infty \) with exponential decay.  
\end{proof}
\begin{theorem}\label{thm:bound_efk}
Suppose that the EFK equation satisfies the Lipschitz condition of Lemma \ref{appendix:EEK_lip}, and the neural network solution \( U\) preserves an energy dissipation law. Moreover, let \( U_0 \in H^{2}_{0}(\mathbf{D}) \), so that the energy dissipation property holds:  
\begin{equation}  
E(U) \leq E(U_{0}).
\end{equation}  
Then, the solution \( U\) is bounded in the \( L^{\infty} \)-norm. 
\begin{proof} This theorem can be proved with the help of \ref{appendix:uniform_bound}. 
\end{proof}
\end{theorem}
\subsubsection{Convergence of multilayer neural network}
This section establishes \( L^2 \) bounds and analyzes the convergence of the multilayer neural network \(U^n\) for both models. To establish the convergence of equations, we follow the approach used for the Cahn-Hilliard equation in \cite{Zhang2023}. From \ref{sec:forward}
\begin{equation} \label{eq:main_neural}
\mathcal{D}(U^{n}) = \mathbf{f}.  
\end{equation}
Additionally, assume that for each \( \mathbf{f} \in Y^{\ast} \), a unique solution \( u \in X^{\ast} \) exists for \eqref{eq:pde}, subject to approximate boundary and initial conditions given by  
\begin{equation}  
\mathcal{B}(U^{n}) = U_{b}^{n} \quad \text{on } \partial \mathbf{D} = \partial \mathbf{D_1}~ \text{or}~ \partial \mathbf{D_2}, \quad U^{n}(0, x) = U^{n}_{0} \quad \text{on } \mathbf{D}= \mathbf{D_1}~ \text{or} ~\mathbf{D_1}.  
\end{equation}  
Here, \( \mathcal{B} \) represents a boundary operator, \( U_{b}^{n} \) is the prescribed boundary data, and \( U^{n}_{0} \) denotes the initial condition.
\begin{theorem}\label{thm:bur_converge}
Let \( U^n_0 \in H^1_0(\mathbf{D_1}) \) be the initial neural network approximation of the Burgess equation. Under the assumptions of lemma \ref{appendix:buress_unique}, there exists a unique solution  
\( u \in H^1(\mathbf{D_1}) \cap H^2(\mathbf{D_1}) \) to the Burgess equation.  
Assume that the Burgess equation satisfies the Lipschitz condition given in \ref{appendix:buress_lip}, and that the sequence \( \{U^n\} \) is uniformly bounded in \( L^2([0,T]; H^1(\mathbf{D_1})) \). Then, the approximate solution \( U^n \) satisfies the following properties:  
\begin{enumerate}
    \item \textbf{Strong convergence in \( L^2 \)}: \( U^n \to u \) strongly in \( L^2(\mathbf{D_1}) \).  
    \item \textbf{Uniform convergence}: \( U^n \) converges uniformly to \( u \) in \( \mathbf{D_1} \).  
\end{enumerate}
Suppose \( U^n \) satisfies the PDE in a bounded domain \( \mathbf{D_1} \) with homogeneous Dirichlet boundary conditions:  
\begin{equation}
\frac{\partial U^n}{\partial t} = \frac{1}{2} \Delta U^n + R(U^n),
\end{equation}  
where the reaction term \( R(U^n) \) satisfies one of the following conditions:  
\begin{enumerate}
    \item \textbf{Linear Growth Condition:}  
    \begin{equation}
    |R(U^n)| \leq C(1 + |U^n|), \quad \text{for some constant } C > 0.
    \end{equation}  
    Under this condition, there exists a constant \( M > 0 \) such that:  
    \begin{equation}
    \sup_{t \in [0,T]} \| U^n(t) \|_{L^2(\mathbf{D_1})} \leq M.
    \end{equation}  

   \item \textbf{Exponential decay condition:}  
    \begin{equation}
    |R(U^n)| \leq C e^{-\lambda |U^n|}, \quad \text{for some constants } C > 0, \lambda > 0.
    \end{equation}  
    
    In this case, \( U^n \) exhibits moderate decay properties, ensuring:  
    \begin{itemize}
        \item Boundedness in \( L^2(\mathbf{D_1}) \),
    \end{itemize}
\end{enumerate}
Under the above assumptions, the sequences \( \{U^n\} \) and \( \{\Delta U^n\} \) remain uniformly bounded in \( L^2(\mathbf{D_1}) \) and \( H^2(\mathbf{D_1}) \), respectively.
\end{theorem}
\begin{proof}

We can prove this theorem following the approach in \cite{Zhang2023}.
\end{proof}
\begin{theorem}\label{thm:efk_converge}
Under the assumptions of Lemma \ref{appendix:EFK_uniqe} and nonlinear term \(F\) satisfies \ref{appendix:EEK_lip}, there exists a unique solution  
\( u \in H^2( \mathbf{D_2}) \cap H^4( \mathbf{D_2}) \) to the  EFK equation \ref{eq:eqmain}. Moreover, if the sequence \( \{U_n\} \) is uniformly bounded and equicontinuous, then the neural network approximation \( U_n \) converges strongly to \( u \) in \( L^2(\mathbf{D_2}) \). Furthermore, \( U_n \) uniformly converges to \( u \) in \( \mathbf{D_2} \).
\end{theorem}

\begin{proof}
We can prove this theorem following the approach in \cite{Zhang2023}.
\end{proof}
\begin{remark}
To achieve accurate learning, it is necessary to employ a sufficiently large number of training (collocation) points. The quadrature error, which is influenced by both the number of collocation points \(N\) and the associated quadrature constants, can be minimized by selecting a large enough \(N\). While an \emph{a priori} estimate for the training error is not available, it can be evaluated once training is complete. The theoretical analysis indicates that, provided the relevant constants remain finite and the RW-PINN is trained appropriately, the resulting relative error will be lower. This is consistent with established principles in machine learning, where a well-trained and suitably regularized RW-PINN \(u^*\) maintains stability and ensures a bounded generalization error. In this work, we set \(N_{\text{int}} > 128\) and \(N_{\text{sb}/\text{tb}/\boldsymbol{d}} > 64\).    
\end{remark}
\begin{remark}
If the residual weight is set to \(1\), the proposed RW-PINN framework reduces to the standard PINN formulation. For linear PDEs, the performance of the RW-PINN is almost equivalent to that of the standard PINN.
\end{remark}

\section{Numerical Experiments} \label{sec:3}
The Puprposed RW-PINN algorithms (\ref{alg1}) and (\ref{alg2}) were implemented using the PyTorch framework \cite{paszke2017automatic}. All numerical experiments were conducted on an Apple Mac-Book equipped with an M3 chip and 24 GB of RAM.
Several hyper-parameters play a crucial role in the PINN framework, including the number of hidden layers \(K - 1\), the width of each layer, the choice of activation function \( \sigma \), the weighting parameter \(\lambda\) in the loss function, the regularization coefficient \(\lambda_{\text{reg}}\) in the total loss and the optimization algorithm for gradient descent. The activation function \( \sigma \) is chosen as the hyperbolic tangent (\(\tanh\)), which ensures smoothness properties necessary for theoretical guarantees in neural networks.
To enhance convergence, the second-order LBFGS optimizer is employed. For optimizing the remaining hyper-parameters, an ensemble training strategy is used, following the methodology in \cite{bai2021physics, mishra2021physics, mishra2022estimates, mishra2023estimates, Murari2024}. This approach systematically explores different configurations for the number of hidden layers, layer width, parameter \(\lambda\), and regularization term \(\lambda_{\text{reg}}\), as summarized in Table~\ref{table_1}. Each hyper-parameter configuration is tested by training the model \(n_\theta\) times in parallel with different random weight initializations. The generalized error, relative generalized error and training loss are denoted as  \( \mathcal{E}_G \),  \( \mathcal{E}^{r}_{G} \) and \( \mathcal{E}_{T} \), respectively. The configuration that achieves the lowest training loss is selected as the optimal model. Numerical experiments have been conducted with a maximum of 5000 LBFGS iterations. 

The forward problems for both models are discussed as follows:

\subsubsection{1D nonlinear Burgess equation:}\label{Example2}
The nonlinear Burgess discussed in \cite{nayied2023numerical} is considered:
\begin{equation}\label{eq:Test2}
\frac{\partial u(t,x)}{\partial t} = \frac{1}{2} \frac{\partial^2 u(t,x)}{\partial x^2} + R(t,x),
\end{equation}
subject to the conditions:
\begin{equation}
u(0,x) = \log(x + 2), \quad u(t,0) = \log(t + 2), \quad u(t,1) = \log(t + 3),
\end{equation}
where
\begin{equation}
R(t,x) = e^{-u(t,x)} + \frac{1}{2} e^{-2u(t,x)}.
\end{equation}
The exact solution is:
\begin{equation}
u(t,x) = \log(x + t + 2).
\end{equation}
Figure \ref{fig:Case2} showcases a graphical comparison between the approximate solutions obtained using PINN and the exact solution for the model given by Eq. \eqref{eq:Test2}. The results demonstrate that the PINN-based approximation remains highly consistent with the exact solution, highlighting its stability. Moreover, Fig. \ref{fig:Case2} clearly depicts the increase in tumor cell density as time \( t \) progresses. A three-dimensional visualization comparing the PINN and exact solutions is presented in Fig. \ref{figures/Ex2b.pdf}. Additionally, Table \ref{table_3} provides the error \( \mathcal{E}_G \) and training error \( \mathcal{E}_{T} \) along with the chosen hyper-parameters. A zoom view of the plot at \( t = 0.5 \) reveals that the PINN prediction aligns more closely with the exact solution compared to the Fibonacci and Haar wavelet methods \cite{nayied2023numerical}.

\begin{figure}[htbp]
\centering
\begin{subfigure}[b]{0.4\textwidth}
    \includegraphics[width=\textwidth, height=0.3\textheight]{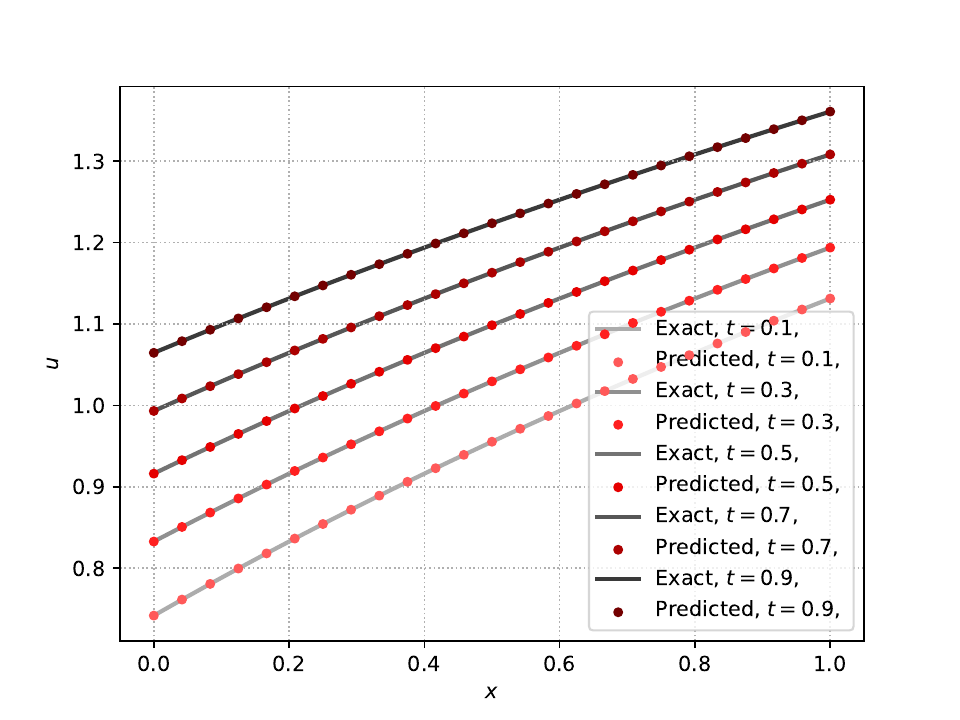}
    \caption{Exact and predicted solution at Different $t$.}
    \label{Ex2}
\end{subfigure}
\hfill
\begin{subfigure}[b]{0.4\textwidth}
    \includegraphics[width=\textwidth, height=0.3\textheight]{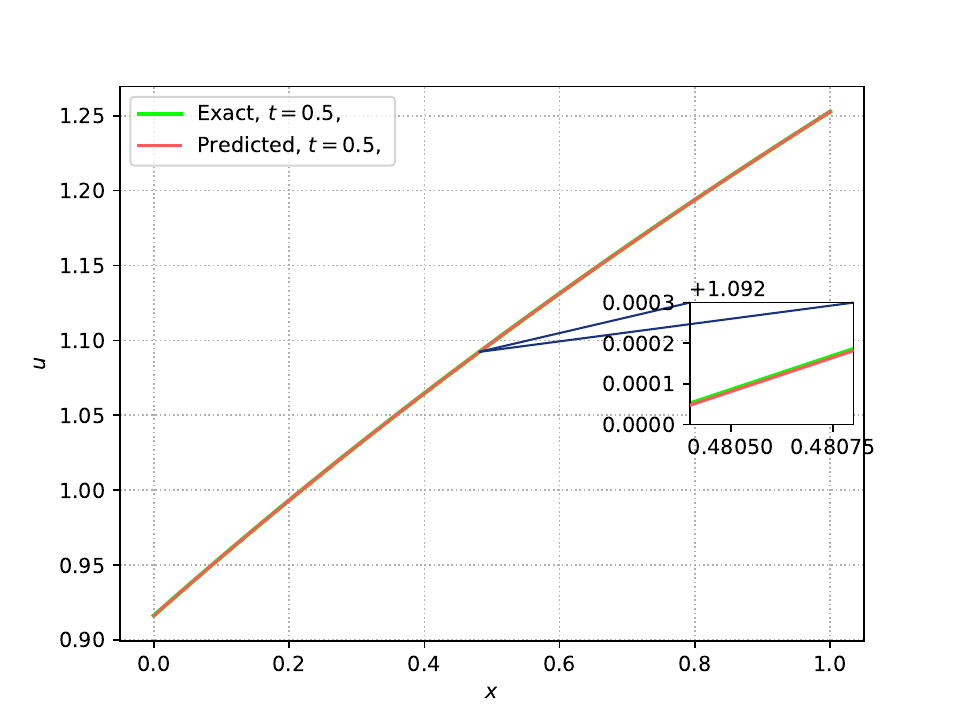}
    \caption{Exact and predicted solution at different $t=0.5$.}
    \label{Ex2a}
\end{subfigure}

\caption{Exact and predicted solution.}
\label{fig:Case2}
\end{figure}

\begin{figure}[htbp]
\centering
\includegraphics[height=0.4\textheight]{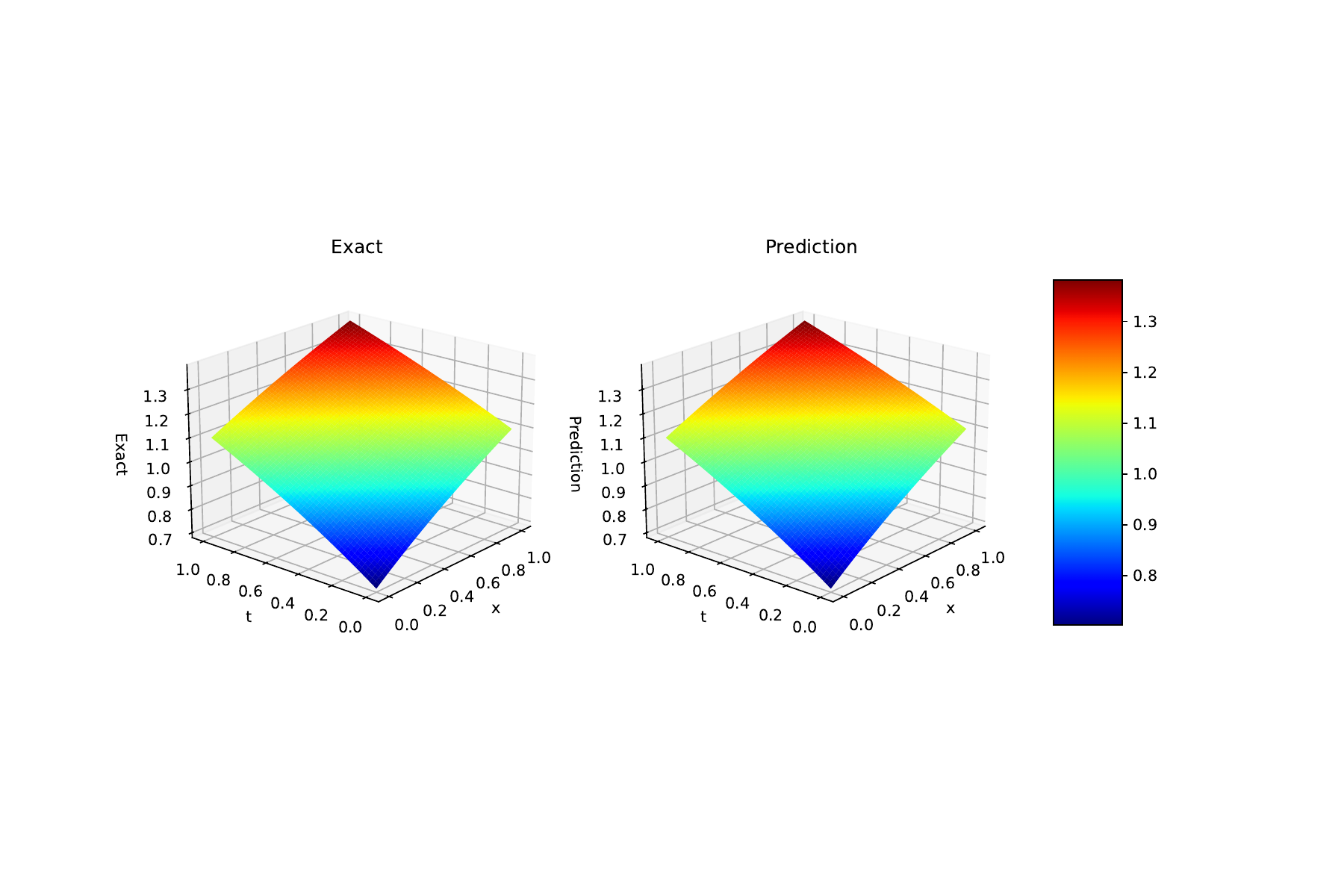}
\caption{Comparison between exact solution and predicted solution.}
\label{figures/Ex2b.pdf}
\end{figure}
\begin{table}[!ht]
\centering
\begin{tabular}{||c c c c c c c c c||}
\hline
Methods & $N_{\text{int}}$ & $N_{\text{sb}}$ & $N_{\text{tb}}$ & $K-1$ & $\bar{d}$ & $\mathcal{E}_{T}$ &  $ \mathcal{E}_G$ & \\ [0.5ex] 
\hline
PINN & 2048 & 512 & 512 & 4 & 20 & 7.1e-05 & 2.3e-05 & \\ 
\hline
RWa PINN & 2048 & 512 & 512 & 4 & 20  & 5.4e-05 & 1.1e-05 &  \\ 
 \hline
RWb PINN  & 2048 & 512 & 512 & 4 & 20  & 3.6e-05 & 6.3e-06 &  \\ 
 \hline
\end{tabular}
\caption{PINN and RW PINN Configuration for Section \ref{Example2}.}
\label{table_3}
\end{table}
\subsubsection{1D nonlinear extended Fisher–Kolmogorov equation}\label{Example3}
The EFK model in one dimension is expressed as follows:
\begin{align}\label{Test4a}
u_{t} + \gamma u_{xxxx} - k_{2} u_{xx} + u^{3} - u &= f, \\
u(0,x) &= \sin(\pi x). \\
u(t,0) = 0, \quad u(t,1) &= 0.
\end{align}
The analytic solution to this model, as presented in \cite{abbaszadeh2020error} (though with different boundary conditions), is given by \( \exp(-t) \sin(\pi x) \). The source term is:
\[
\exp(-t) \sin(\pi x) \left( \gamma \pi^{4} + \pi^{2} - 2 + \exp(-2 t) (\sin(\pi x))^{2}\right).
\]
Figure \ref{fig:Case44} shows a graphical comparison between the approximate solutions obtained using PINN and the exact solution. The results demonstrate that the PINN-based approximation closely matches the exact solution, validating its stability. A three-dimensional visualization comparing the PINN and exact solutions is provided in Fig.\ref{figures/Ex3cc.pdf}. Table \ref{table_4} presents the error \( \mathcal{E}_G \) and training error \( \mathcal{E}_{T} \), along with the selected hyper-parameters.

\begin{figure}[htbp]
\centering
\begin{subfigure}[b]{0.4\textwidth}
    \includegraphics[width=\textwidth, height=0.3\textheight]{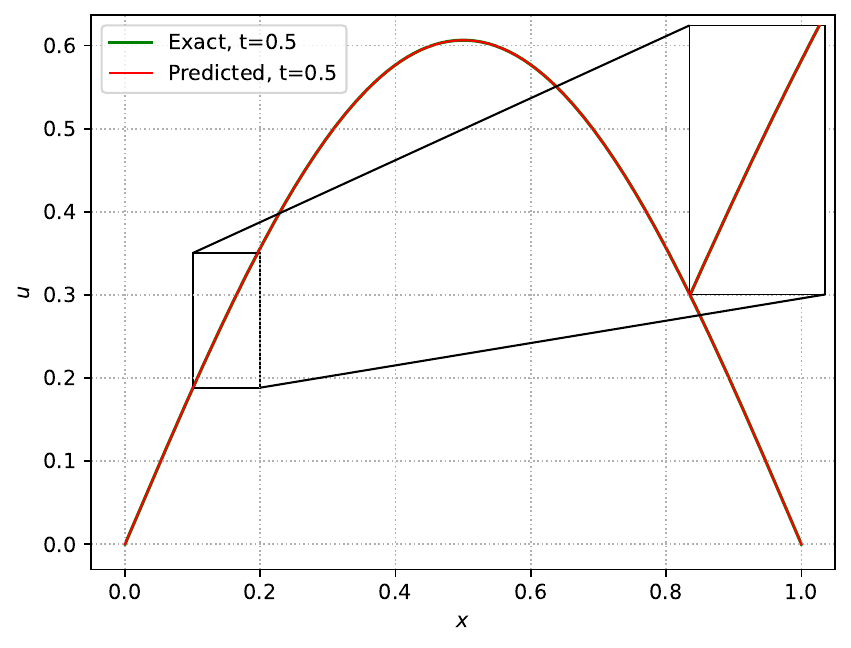}
    \caption{Exact versus predicted solution at different $t=0.5$.}
    \label{Ex3a}
\end{subfigure}
\hfill
\begin{subfigure}[b]{0.4\textwidth}
    \includegraphics[width=\textwidth, height=0.3\textheight]{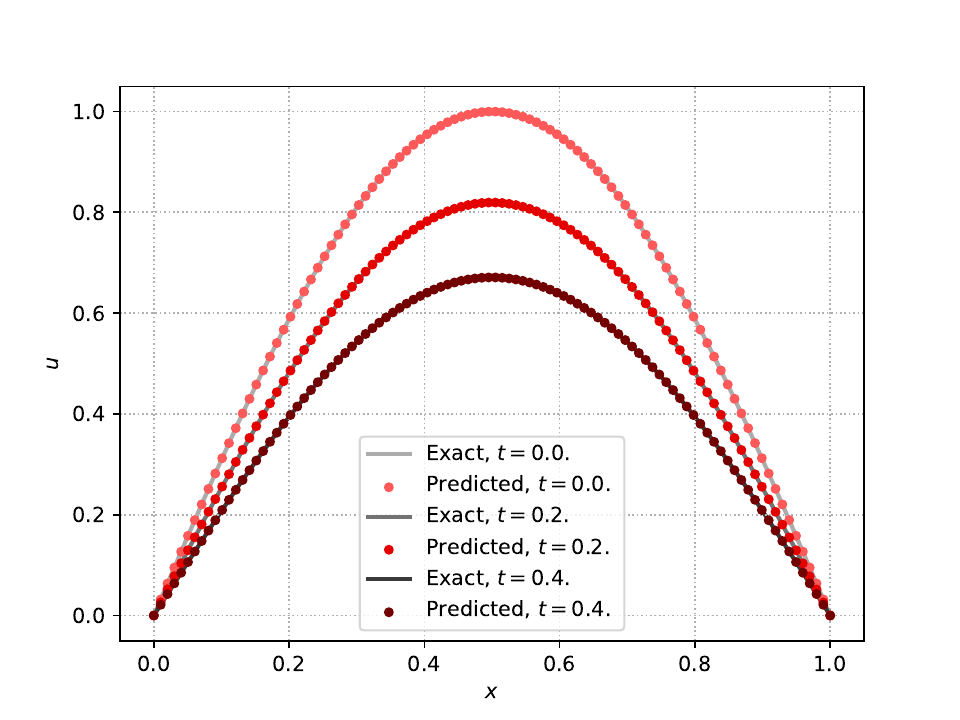}
    \caption{Exact and predicted solution at different t.}
    \label{Ex3b}
\end{subfigure}
\caption{Comparison between exact solution and predicted solution at $\gamma =0.001$.}
\label{fig:Case44}
\end{figure}
\begin{figure}[htbp]
\centering
\includegraphics[height=0.4\textheight]{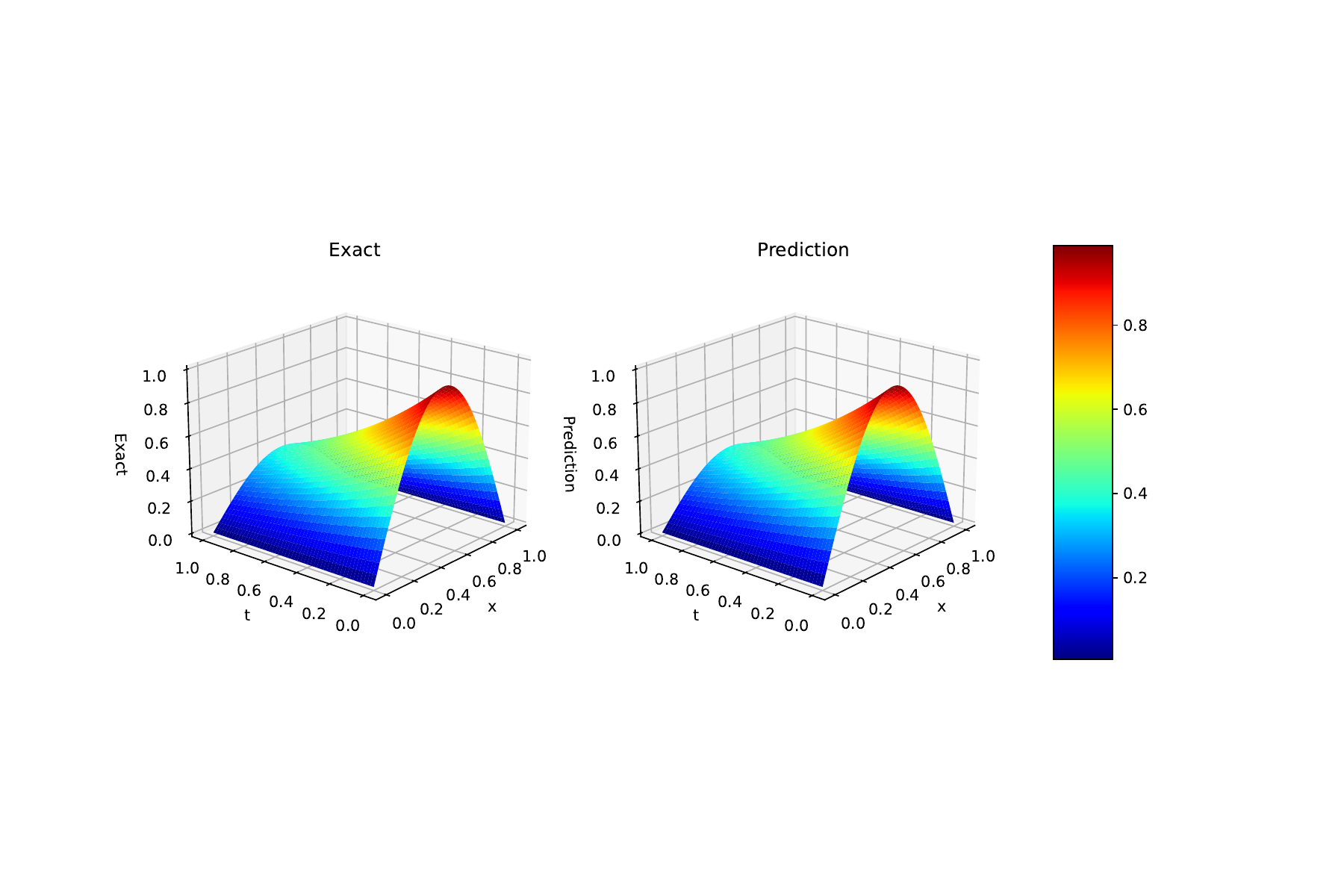}
\caption{Comparison between exact solution and predicted solution at $\gamma =0.001$..}
\label{figures/Ex3cc.pdf}
\end{figure}

\begin{table}[!h]
\centering
\begin{tabular}{||c c c c  c c c c c||}
\hline
Methods & $N_{\text{int}}$  &$N_{\text{sb}}$ & $N_{\text{tb}}$ & $K-1$ & $\bar{d}$  & $\mathcal{E}_{T}$ & $ \mathcal{E}_G$ &\\ [0.5ex] 
\hline
PINN & 4096 & 1024 & 1024 & 4 & 20 & 0.00080 & 1.8e-04 &\\ 
\hline
RWa PINN & 4096 & 1024 & 1024 & 4 & 20 & 0.00055 & 3.5e-05 &\\ 
\hline
RWb PINN & 4096 & 1024 & 1024 & 4 & 20 & 0.00059 & 3.8e-05 &\\ 
\hline
\end{tabular}
\caption{PINN and RW PINN Configuration for Section \ref{Example3}.}
\label{table_4}
\end{table}

\subsubsection{1D extended Fisher–Kolmogorov equation}\label{Example3+1}
Consider the EFK equation:
\begin{equation}
    u_t + \gamma u_{xxxx} - u_{xx} + u^3 - u = 0, \quad (x, t) \in [0,1] \times (0,T],
\end{equation}
with initial and boundary conditions:

\begin{itemize}
    \item[(a)] \( u(x,0) = x^3 (1 - x)^3 \),
    \item[(b)] \( u(x,0) = x^2 (1 - x)^2 \),
\end{itemize}

\begin{equation}
    u(0,t) = 0, \quad u(1,t) = 0,
\end{equation}
\begin{equation}
    u_{xx}(0,t) = 0, \quad u_{xx}(1,t) = 0.
\end{equation}
The numerical solution for this equation has been computed using the parameter \( \gamma = 0.01 \) with different initial values. Figure \ref{Ex3+1a} shows the numerical solution for the initial condition \( u(x,0) = x^3 (1 - x)^3 \), while Figure~\ref{Ex3+1b} corresponds to the initial condition \( u(x,0) = x^2 (1 - x)^2 \). Both figures display the numerical solutions at different times, exhibiting the same characteristics as those presented in \cite{Priyanka2024}.  
Table~\ref{table_3+1} reports the training error \( \mathcal{E}_{T} \) alongside the chosen hyperparameters.

\begin{table}[!h]
\centering
\begin{tabular}{||c c c c c  c c c c||}
\hline
 Intial Conditions & Methods & $N_{\text{int}}$  &$N_{\text{sb}}$ & $N_{\text{tb}}$ & $K-1$ & $\bar{d}$ & $\mathcal{E}_{T}$ & \\ [0.5ex] 
\hline
(a) & PINN & 2048 & 512 & 512 & 4 & 20 & 0.0008 &\\ 
\hline
(a) & RWa PINN  & 2048 & 512 & 512 & 4 & 20 & 0.0005 &\\ 
\hline
(a) & RWb PINN  & 2048 & 512 & 512 & 4 & 20 & 0.0005 &\\ 
\hline
\hline
(b) & PINN & 4096 & 1024 & 1024 & 4 & 20 & 0.0009 & \\ 
\hline
(b) & RWa PINN  & 4096 & 1024 & 1024 & 4 & 20 & 0.0006 & \\ 
\hline
(b) & RWb PINN & 4096 & 1024 & 1024 & 4 & 20 & 0.0006 & \\ 
\hline
\end{tabular}
\caption{Best-Performing PINN Configuration for Section \ref{Example3+1}.}
\label{table_3+1}
\end{table}

\begin{figure}[htbp]
\centering
\begin{subfigure}[b]{0.4\textwidth}
    \includegraphics[width=\textwidth, height=0.3\textheight]{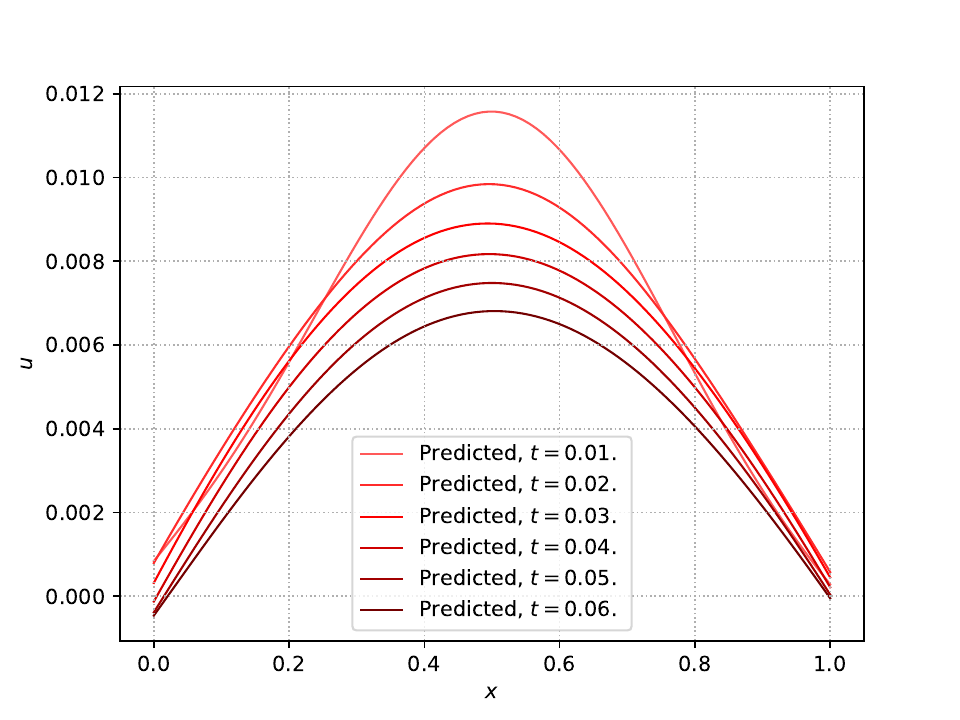}
    \caption{Prediction}
    \label{Ex3+1a}
\end{subfigure}
\hfill
\begin{subfigure}[b]{0.4\textwidth}
    \includegraphics[width=\textwidth, height=0.3\textheight]{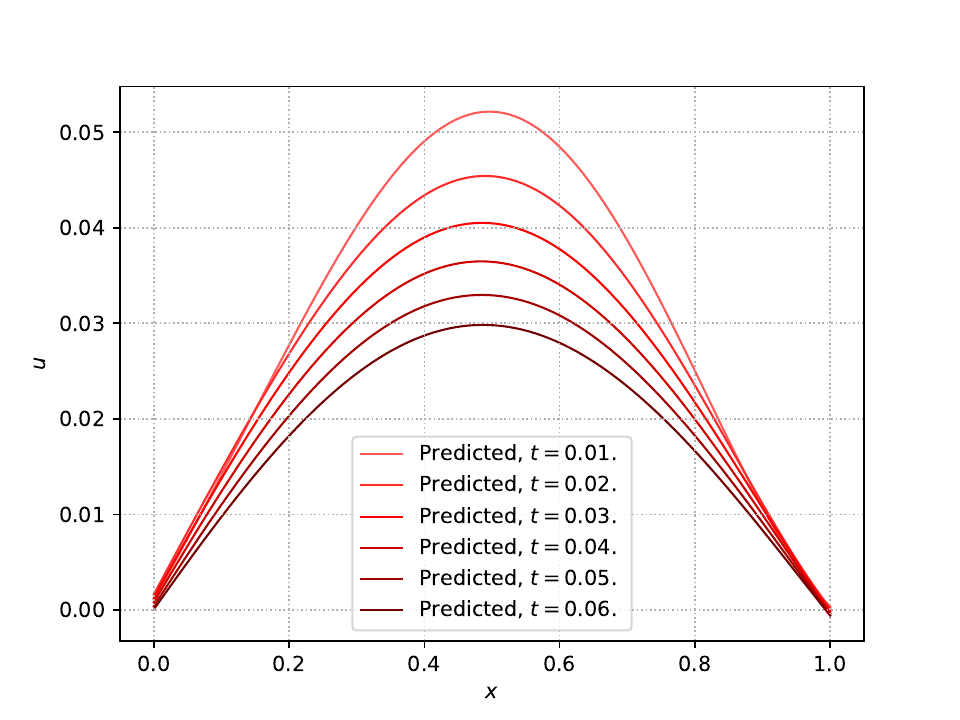}
    \caption{Prediction}
    \label{Ex3+1b}
\end{subfigure}
\caption{Predicted plot with  \( \gamma = 0.01. \)}
\label{fig:Case3+1}
\end{figure}
\subsubsection{2D extended Fisher–Kolmogorov equation}\label{Example4}
In this study, we focus on the 2D nonlinear EFK equation. 

\begin{align} \label{eq5}
u_{t} +  \gamma \Delta^{2} u - \Delta u + u^{3} - u &= g(t,x, y) \quad && \text{in} \quad  (0,T] \times \mathbf{D_2} \\
u(0,x,y) &= \sin(\pi x) \sin(\pi y) \quad && \text{in} \quad \mathbf{D_2} \\
u &= f_{1}, \quad \Delta u = f_{2} \quad && \text{on} \quad (0,T] \times \partial \mathbf{D_2}.
\end{align}

The exact solution to equation (\ref{eq5}) is \( \sin(\pi x) \sin(\pi y) \exp(-t) \). The source term $g$ and the boundary conditions $f_1$ and $f_2$ are obtained directly from the exact solution. The subsequent plots compare the exact and predicted solutions, shown in both contour and 3D surface formats. Figures \ref{fig:Case4} and \ref{fig:Case4a} provide graphical comparisons of the approximate solutions obtained using PINN based algorithms and the exact solution for \( \gamma = 0.01 \), displayed as 3D visualizations at \( t = 0 \) and \( t = 1 \). The results confirm that the PINN-based approximation aligns closely with the exact solution, demonstrating its stability. Additionally, Fig. \ref{fig:Case4a} depicts the contour plot at \( t = 1 \). Table \ref{table_5} reports the error \( \mathcal{E}_G \) and training error \( \mathcal{E}_{T} \), alongside the chosen hyper-parameters.

\begin{figure}[htbp]
\centering
\begin{subfigure}[b]{0.4\textwidth}
    \includegraphics[width=\textwidth, height=0.3\textheight]{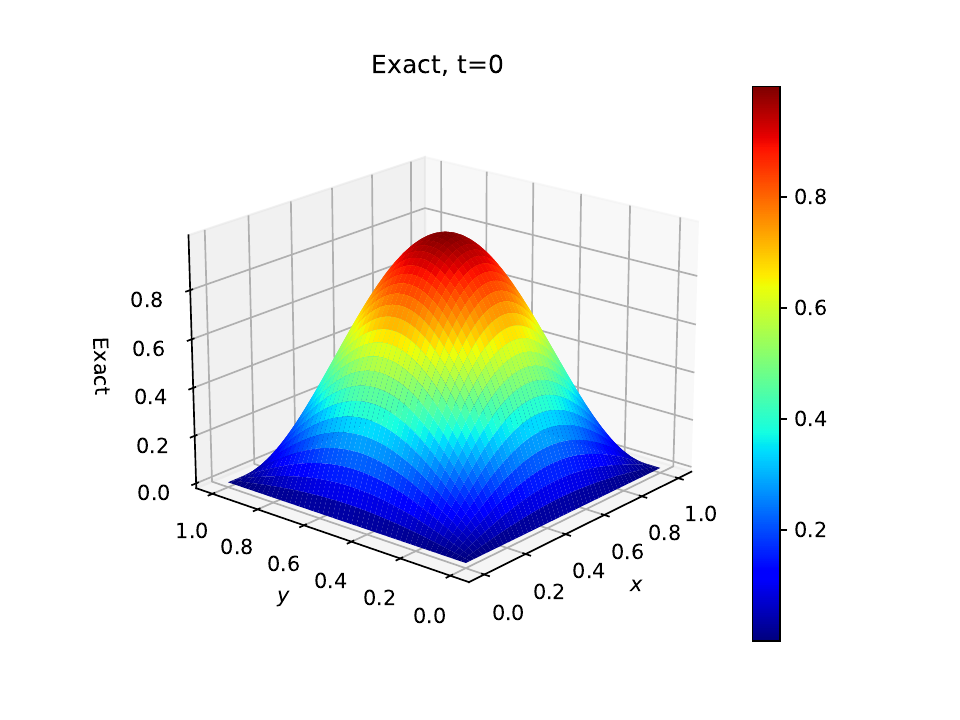}
    \caption{Exact}
    \label{Ex4a0}
\end{subfigure}
\hfill
\begin{subfigure}[b]{0.4\textwidth}
    \includegraphics[width=\textwidth, height=0.3\textheight]{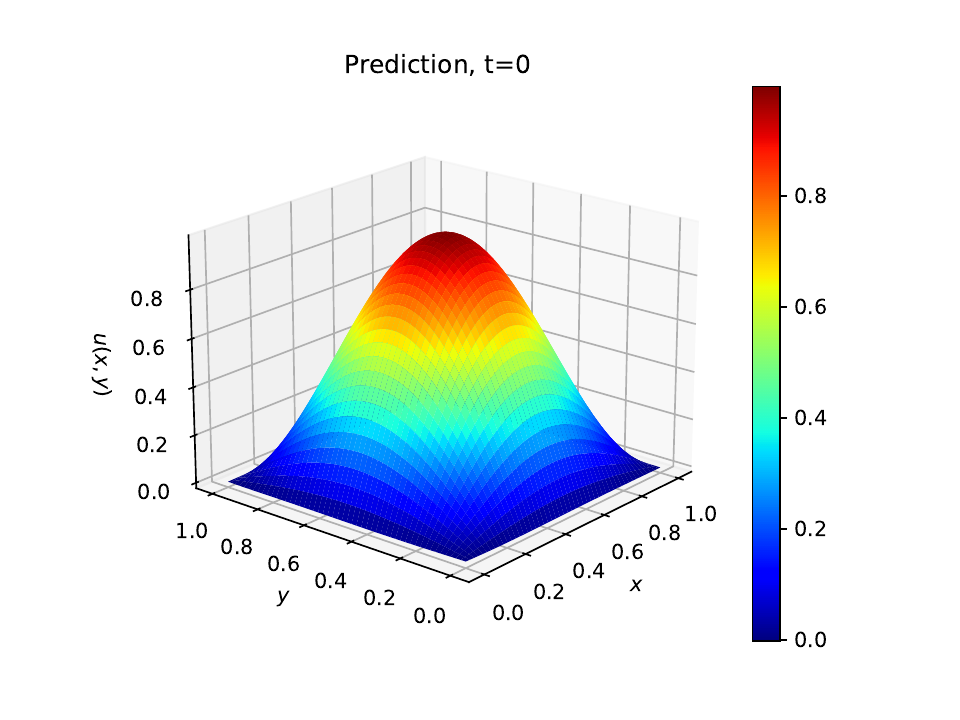}
    \caption{Predicted}
    \label{Ex4a1}
\end{subfigure}
\caption{Exact and predicted contour plot at \( T = 0 \) with \( \gamma = 0.01 \).}
\label{fig:Case4}
\end{figure}

\begin{figure}[htbp]
\centering
\begin{subfigure}[b]{0.4\textwidth}
    \includegraphics[width=\textwidth, height=0.3\textheight]{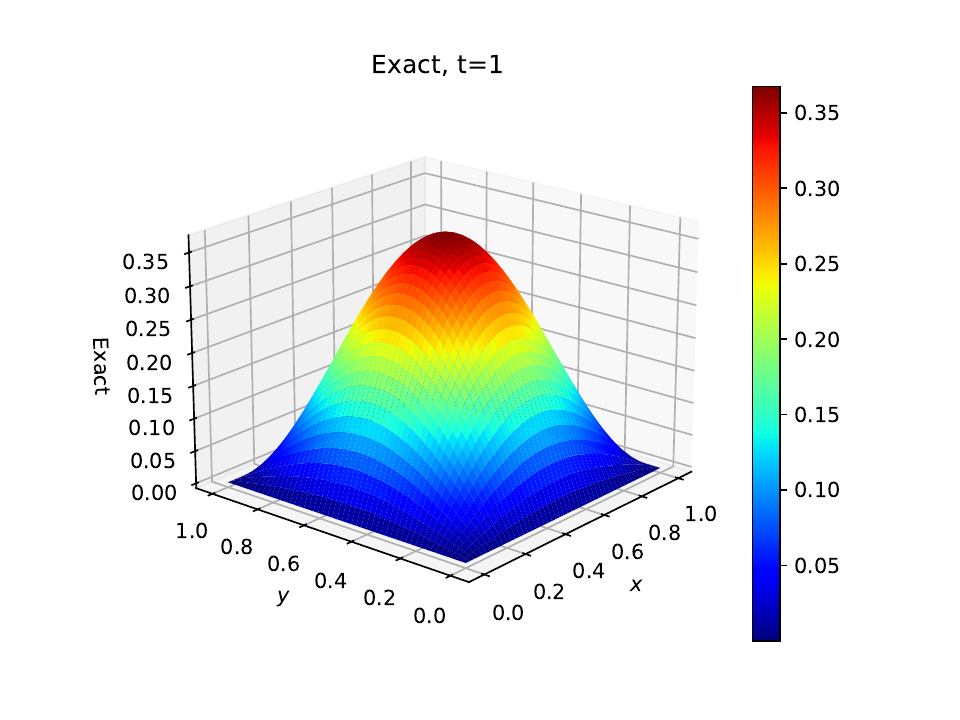}
    \caption{Exact}
    \label{Ex4bex}
\end{subfigure}
\hfill
\begin{subfigure}[b]{0.4\textwidth}
    \includegraphics[width=\textwidth, height=0.3\textheight]{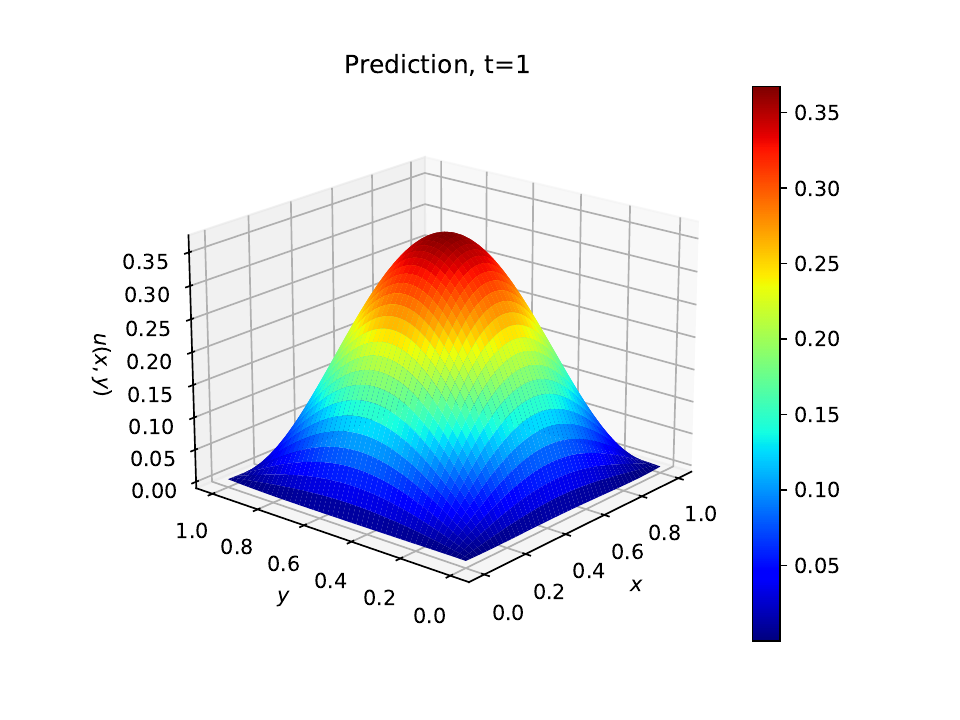}
    \caption{Predicted}
    \label{Ex4bpre}
\end{subfigure}
\caption{Exact and predicted 3D plot at \( T = 1 \) with \( \gamma = 0.01 \).}
\label{fig:Casea}
\end{figure}

\begin{figure}[htbp]
\centering
\begin{subfigure}[b]{0.4\textwidth}
    \includegraphics[width=\textwidth, height=0.3\textheight]{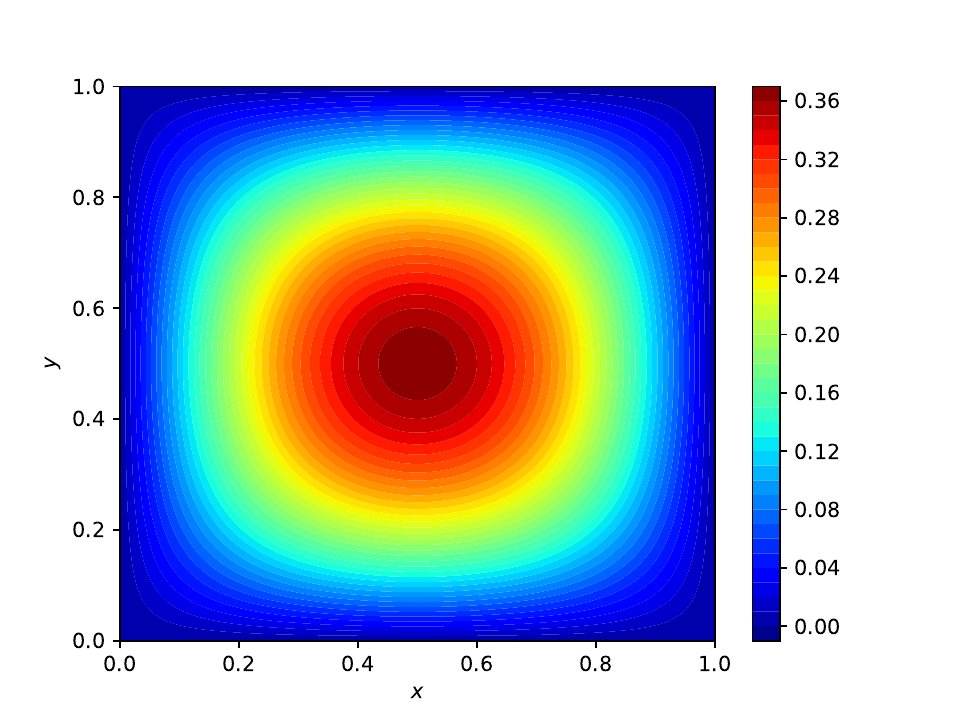}
    \caption{Exact}
    \label{Ex4bex2}
\end{subfigure}
\hfill
\begin{subfigure}[b]{0.4\textwidth}
    \includegraphics[width=\textwidth, height=0.3\textheight]{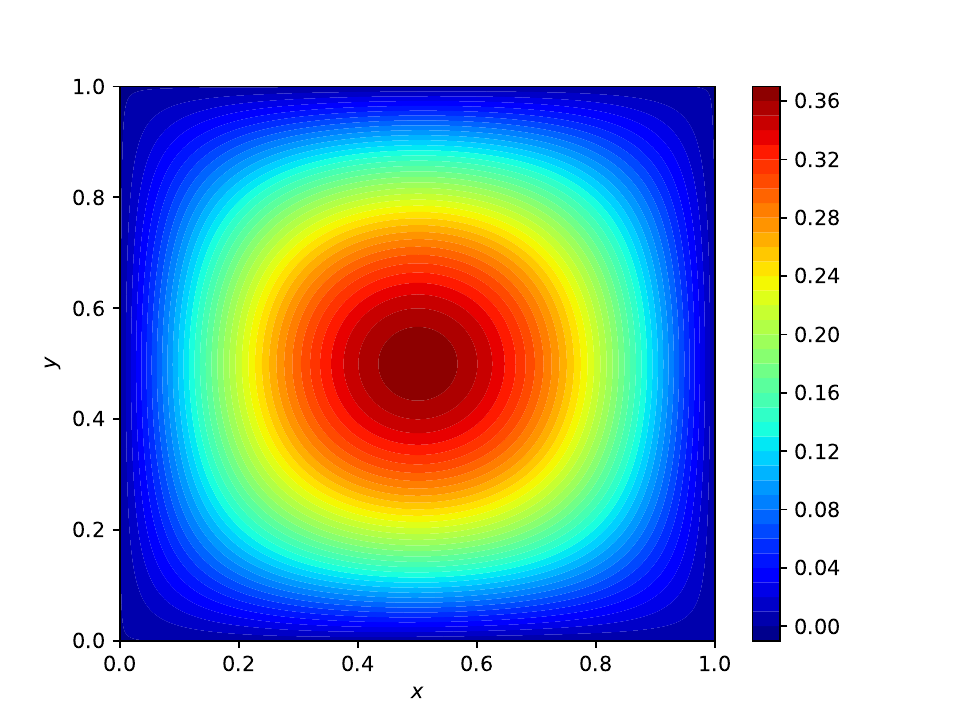}
    \caption{Predicted}
    \label{Ex4bbpre}
\end{subfigure}
\caption{Exact and predicted contour plot at \( T = 1 \) with \( \gamma = 0.01 \).}
\label{fig:Case4a}
\end{figure}

\begin{table}[!ht]
\centering
\begin{tabular}{|| c c c  c c c c c c||}
\hline
 Method & $N_{\text{int}}$ & $N_{\text{tb}}$ & $N_{\text{sb}}$ & $K-1$ & $\bar{d}$ & $\mathcal{E}_{T}$ & $ \mathcal{E}_G$  & \\ [0.5ex] 
\hline
PINN & 8192 & 2048 & 2048 & 4 & 28 & 0.0024 & 1.1e-03  & \\ 
 \hline
RWa PINN & 8192 & 2048 & 2048 & 4 & 28 & 0.0011 & 5.3e-04  & \\ 
 \hline
RWb PINN & 8192 & 2048 & 2048 & 4 & 28 & 0.0013 & 4.6e-04  &  \\ 
\hline
\end{tabular}
\caption{PINN and RW PINN Configuration for Section \ref{Example4}.}
\label{table_5}
\end{table}

\subsection{Inverse Problems}
The inverse problems for both models are discussed as follows:
\subsubsection{1D nonlinear Burgess equation}\label{Example5}
The nonlinear Burgess equation considered in \cite{nayied2023numerical} is given as follows:
\begin{equation}\label{eq:Test5}
\frac{\partial u}{\partial t} = \frac{1}{2} \frac{\partial^2 u}{\partial x^2} + R(t,x),
\end{equation}
where
\begin{equation}
R(t,x) = e^{-u(t,x)} + \frac{1}{2} e^{-2u(t,x)}.
\end{equation}
The exact solution is:
\begin{equation}
u(t,x) = \log(x + t + 2).
\end{equation}

The results confirm that the PINN-based approximation closely matches the exact solution, demonstrating its stability. 

The plot resembles Example~\ref{Example2}.
Additionally, Table \ref{table_5} reports the error \( \mathcal{E}_G \) and training error \( \mathcal{E}_{T} \) along with the selected hyper-parameters.

\begin{table}[!ht]
\centering
\begin{tabular}{||c c  c c c c||}
\hline
Method & $N$ & $K-1$ & $\bar{d}$ & $\mathcal{E}_{T}$ & $ \mathcal{E}_G$   \\  [0.5ex] 
\hline
PINN & 3072 & 4 & 20 & 4.7e-05 & 1.9e-05 \\ 
\hline
RWa PINN & 3072 & 4 & 20 &  5.3e-05 & 1.7e-05  \\
\hline
RWb PINN & 3072 & 4 & 20 & 3.7e-05 & 9.8e-06  \\
\hline
\end{tabular}
\caption{PINN and RW PINNs Configuration for Section \ref{Example5}.}
\label{table_5}
\end{table}

\subsubsection{1D extended Fisher–Kolmogorov equation}\label{Example6}
The 1D case of the EFK model \cite{abbaszadeh2020error} is given as follows:
\begin{align}\label{Test6}
u_{t} + \gamma u_{xxxx} - u_{xx} + u^{3} -u  = f, \\
\end{align}
The exact solution is :  
\begin{equation}\label{eq:Test6a_sol}
u(t,x) = \exp{(-t)}\sin(\pi x).
\end{equation}
The corresponding source term is:  
\[
f(t,x) = \exp{(-t)} \sin(\pi x) \left( \gamma  \pi^{4} + \pi^{2} - 2 + \exp(-2t) (\sin(\pi x))^{2} \right).
\]
Figure \ref{fig:Case6} presents a graphical comparison of the approximate solutions obtained using PINN based algorithms and the exact solution for the model represented by Eq.\eqref{eq:Test6a_sol}. The results indicate that the PINN-based approximation aligns closely with the exact solution, validating its stability. Fig.~\ref{figures/Ex6a.pdf} further depicts the variation of the EFK equation over time \( t \). A three-dimensional visualization comparing the PINN and exact solutions is provided in Fig. \ref{figures/Ex6aa.pdf}. Furthermore, Table \ref{table_6} presents the  error \( \mathcal{E}_G \) and training error \( \mathcal{E}_{T} \) alongside the chosen hyper-parameters.

\begin{figure}[htbp]
\centering
\begin{subfigure}[b]{0.45\textwidth}
    \includegraphics[width=\textwidth, height=0.3\textheight]{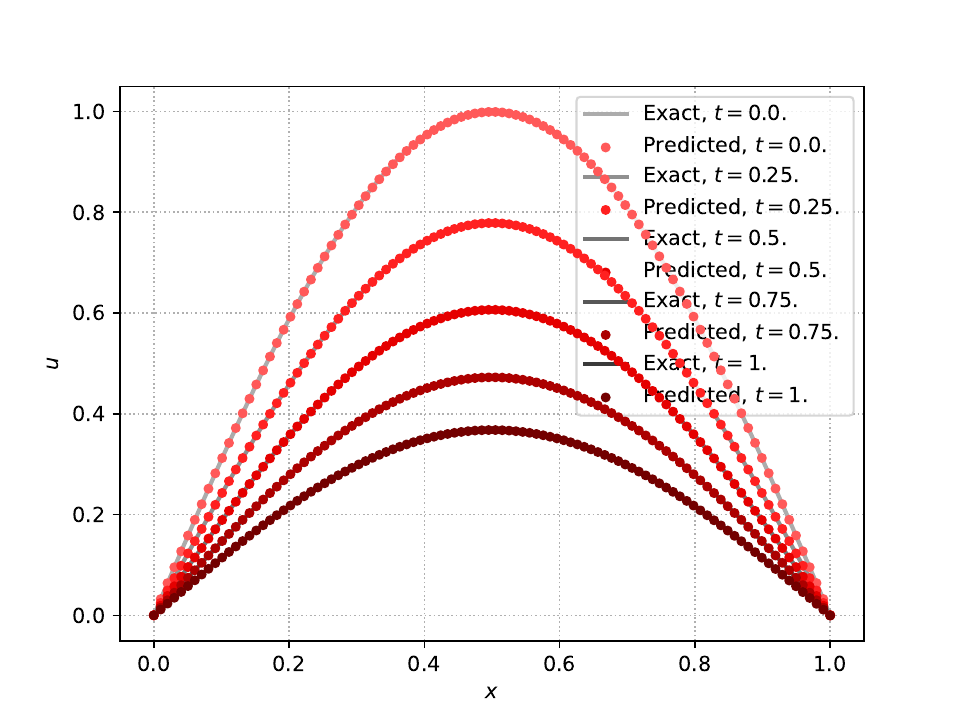}
    \caption{Exact and Predicted solution at different t}
    \label{figures/Ex6a.pdf}
\end{subfigure}
\hspace{0.05\textwidth}  
\begin{subfigure}[b]{0.45\textwidth}
    \includegraphics[width=\textwidth, height=0.35\textheight]{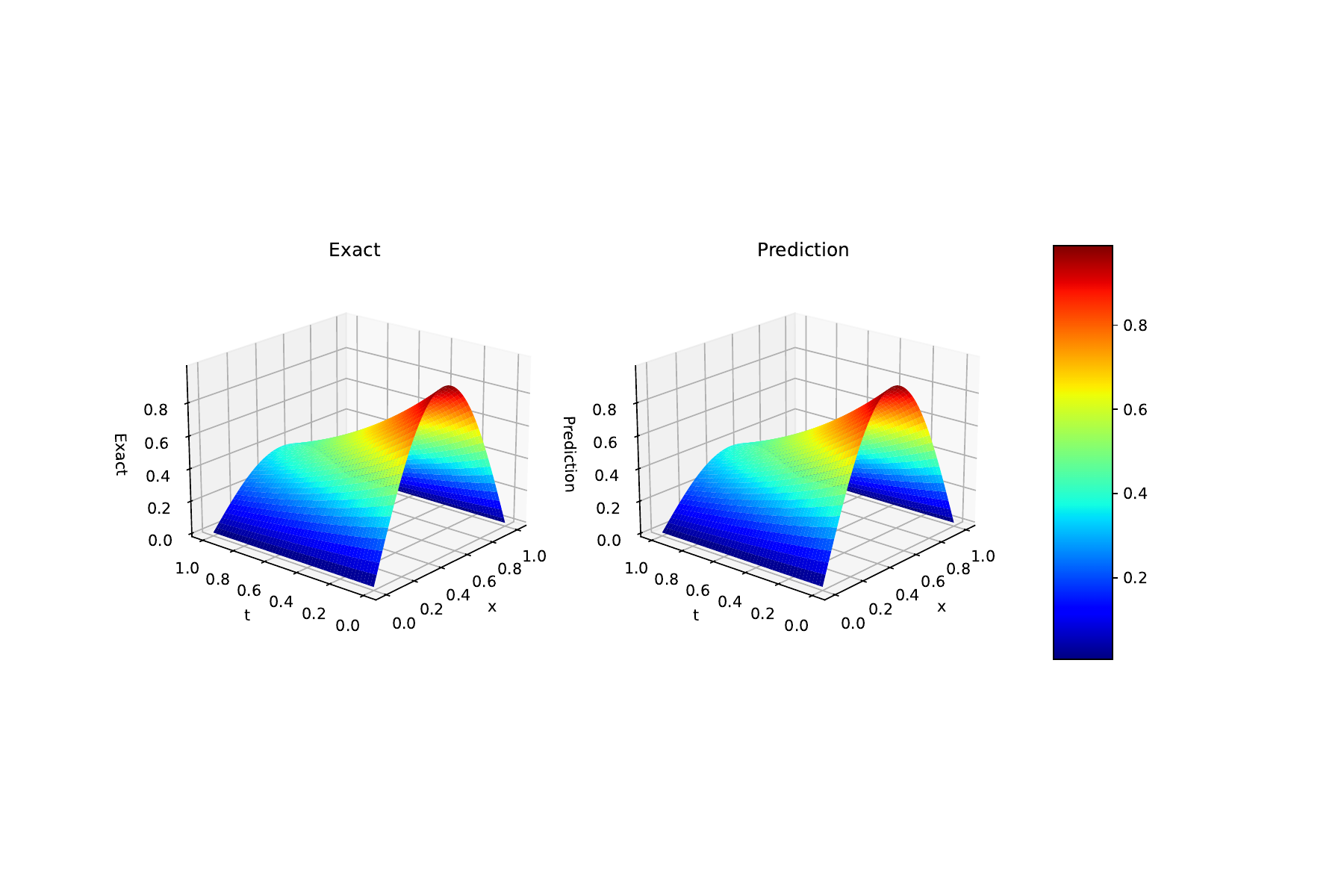}
    \caption{Exact and Predicted surface plot}
    \label{figures/Ex6aa.pdf}
\end{subfigure}
\caption{Exact and predicted solution of model at $\gamma =0.001$.}
\label{fig:Case6}
\end{figure}

\begin{table}[!ht]
\centering
\begin{tabular}{||c  c c  c c c c||}
\hline
 Method & $N $ & $K-1$ & $\bar{d}$ & $\mathcal{E}_{T}$ & $ \mathcal{E}_G$  & \\ [0.5ex] 
\hline
PINN & 6144 & 4 & 20  & 0.0008 & 1.0e-04  & \\ 
\hline
RWa PINN & 6144 & 4 & 20  & 0.0005 & 2.7e-05  & \\ 
\hline
RWb PINN & 6144 & 4 & 20 & 0.0005 & 7.4e-05  & \\ 
\hline
\end{tabular}
\caption{PINN and RW PINNs Configuration for Section \ref{Example6}.}
\label{table_6}
\end{table}

\subsubsection{2D extended Fisher–Kolmogorov equation:}\label{Example7}
The 2D equation  has following exact solution \cite{al2024finite }\cite{ilati2018direct} :
\[ \exp({-t}) \exp\left(-\frac{(x-0.5)^{2}}{\beta} - \frac{(y-0.5)^{2}}{\beta}  \right). \]
The source term is derived from the exact solution. The model is solved for different parameter values over time. Both the exact and predicted solutions are presented in contour and 3D surface formats, as shown in the following sub-figures.
 Figures \ref{fig:Case7a}, \ref{fig:Case7b}, and \ref{fig:Case7c} present a graphical comparison of the approximate solutions obtained using PINN and the exact solution for varying values of \( \beta \), displayed as contour plots and 3D visualizations. The results show that the PINN-based approximation closely matches the exact solution, confirming its stability. Additionally, Fig. \ref{figures/Ex6a.pdf} illustrates the changes in tumor cell density over time \( t \). Table \ref{table_8} reports the error \( \mathcal{E}_G \) and training error \( \mathcal{E}_{T} \) along with the selected hyper-parameters.

\begin{figure}[htbp]
\centering
\begin{subfigure}[b]{0.4\textwidth}
    \includegraphics[width=\textwidth]{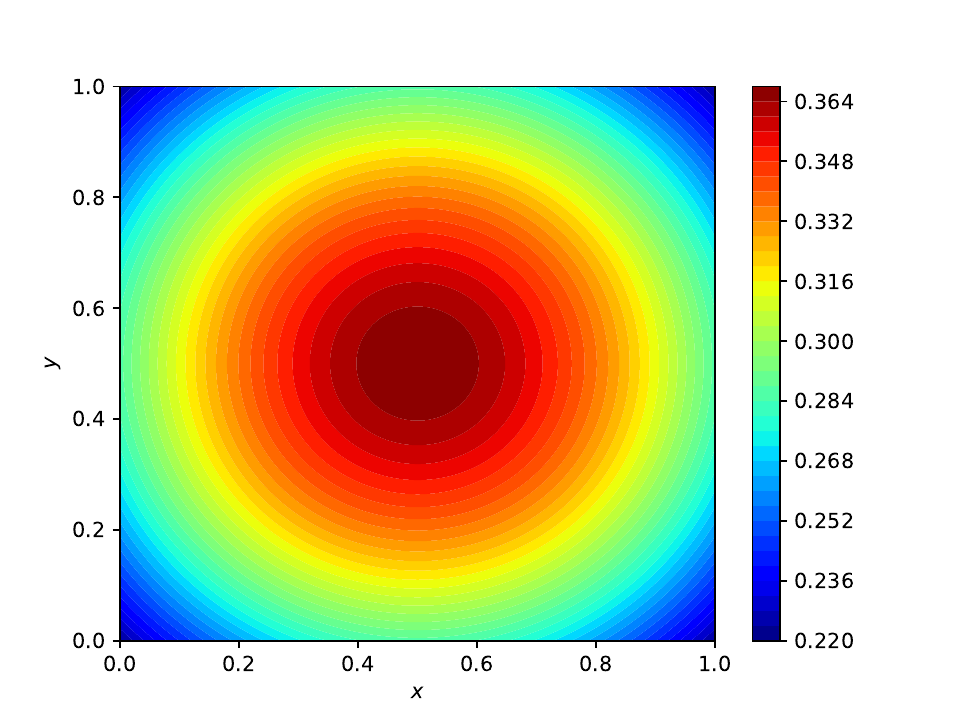}
    \caption{Exact }
    \label{Ex7a_exact1}
\end{subfigure}
\hfill
\begin{subfigure}[b]{0.4\textwidth}
    \includegraphics[width=\textwidth]{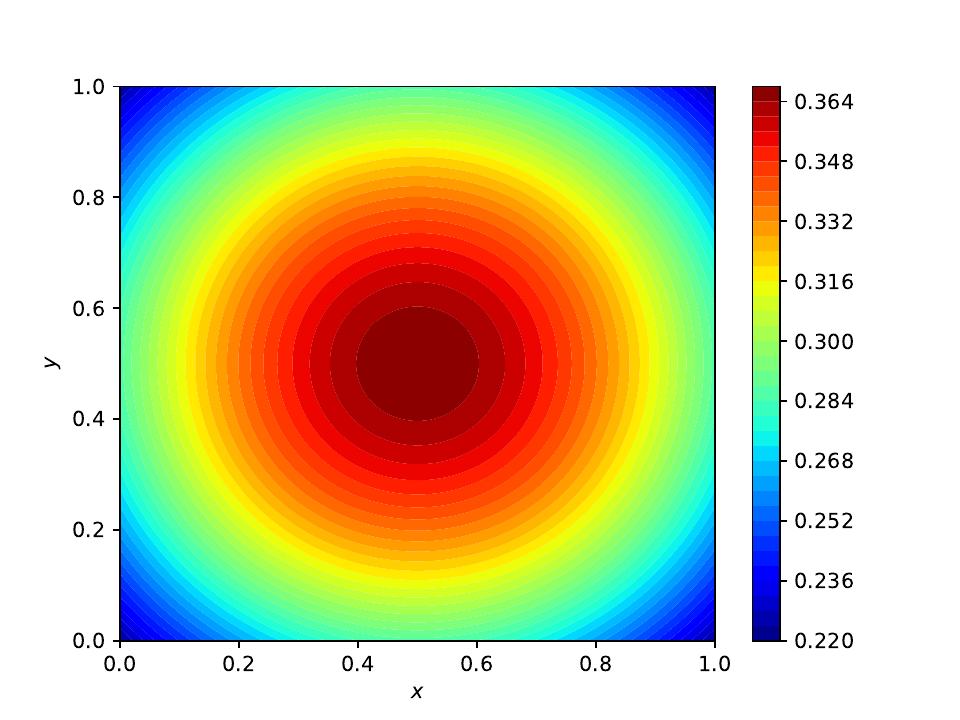}
    \caption{Predicted}
    \label{Ex7a_predicted1}
\end{subfigure}

\begin{subfigure}[b]{0.4\textwidth}
    \includegraphics[width=\textwidth]{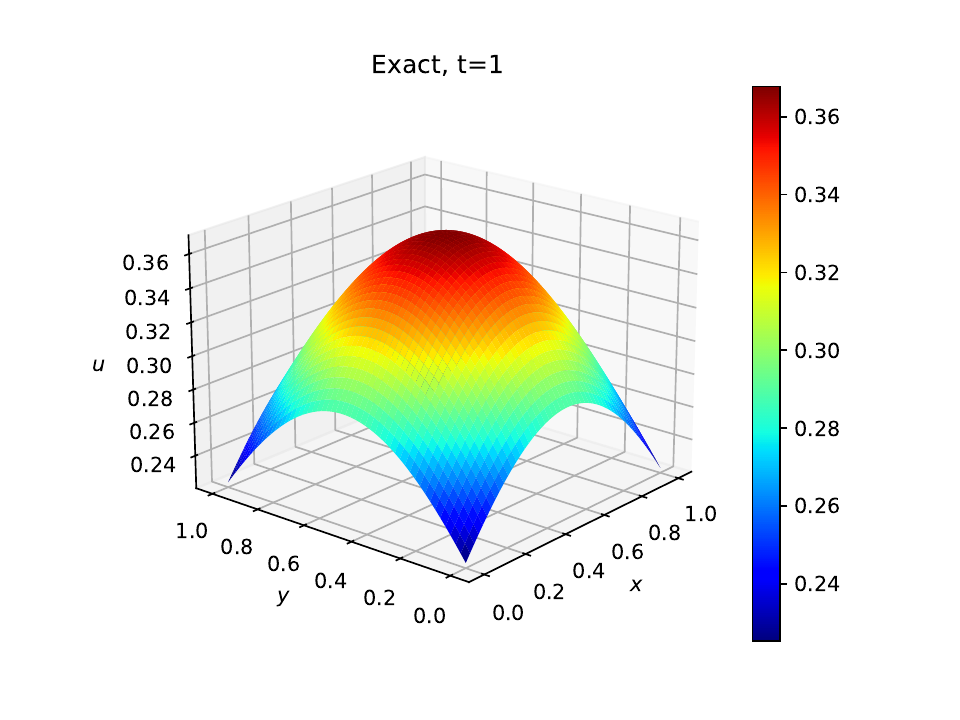}
    \caption{Exact }
    \label{Ex7aa_exact1}
\end{subfigure}
\hfill
\begin{subfigure}[b]{0.4\textwidth}
    \includegraphics[width=\textwidth]{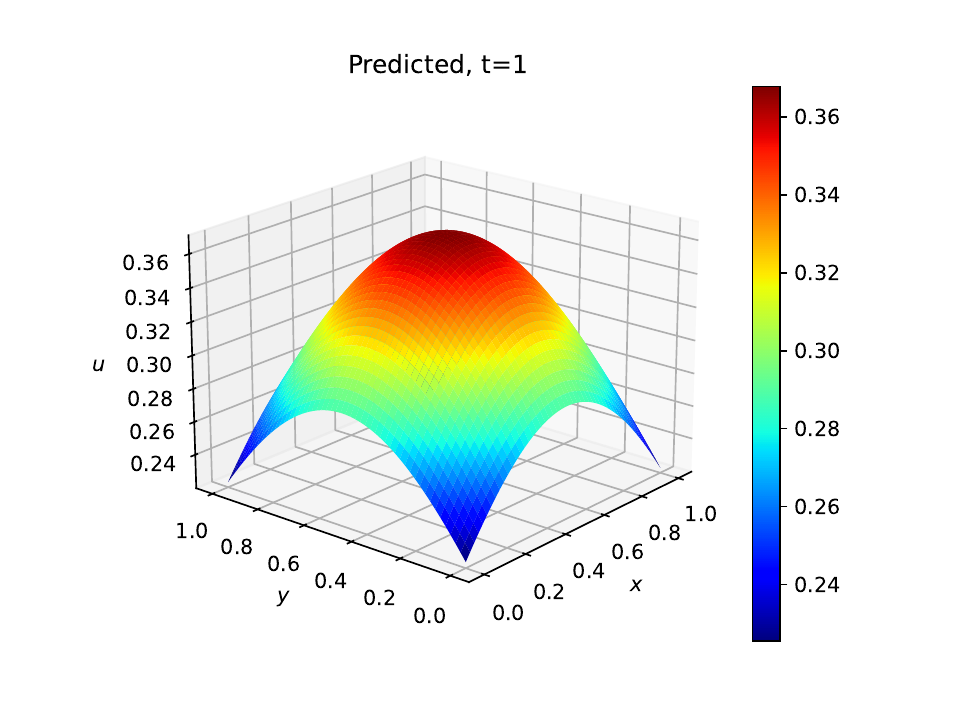}
    \caption{Predicted }
    \label{Ex7aa_predicted1}
\end{subfigure}

\caption{The exact and predicted solutions at \( T = 1 \), with  \( \gamma = 0.0001 \) and \( \beta = 1 \).}
\label{fig:Case7a}
\end{figure}

\begin{figure}[htbp]
\centering
\begin{subfigure}[b]{0.4\textwidth}
    \includegraphics[width=\textwidth]{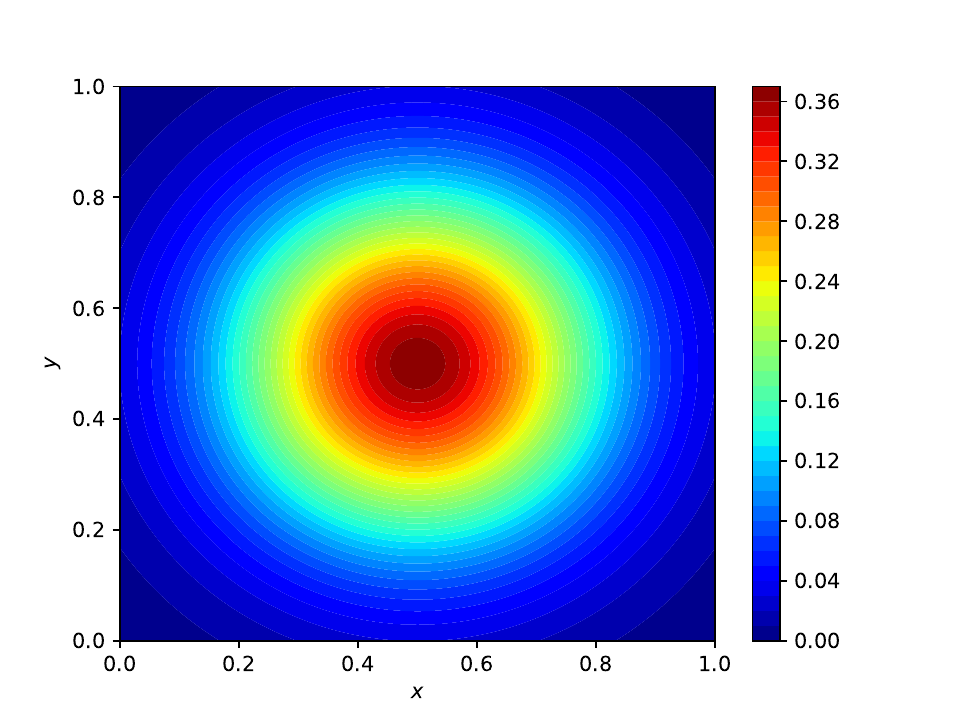}
    \caption{Exact }
    \label{Ex7b_exact1}
\end{subfigure}
\hfill
\begin{subfigure}[b]{0.4\textwidth}
    \includegraphics[width=\textwidth]{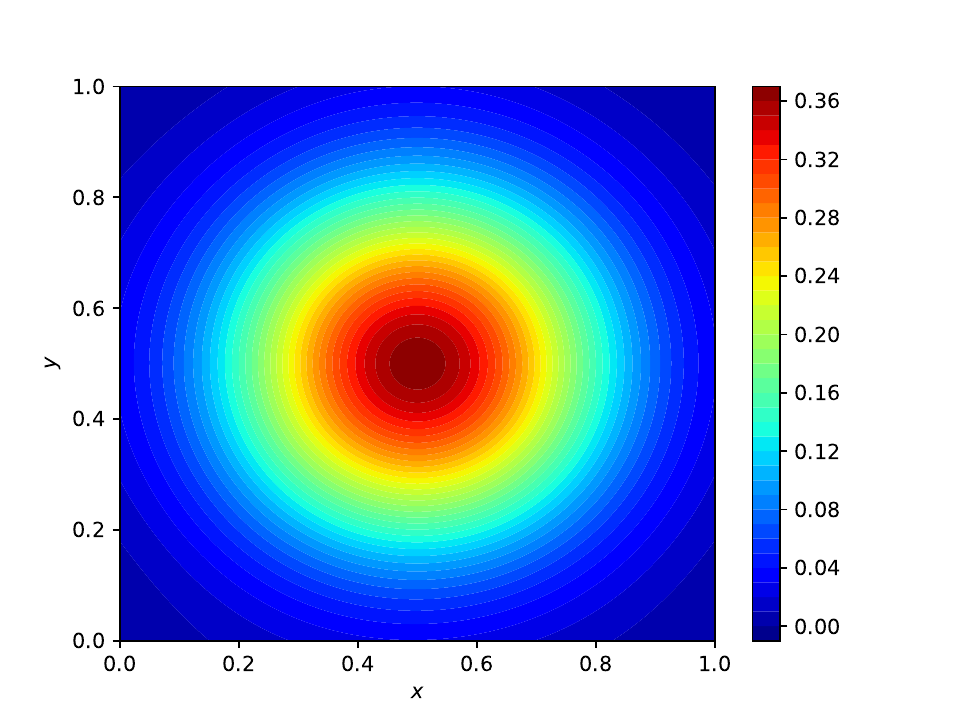}
    \caption{Predicted }
    \label{Ex7bpre}
\end{subfigure}

\begin{subfigure}[b]{0.4\textwidth}
    \includegraphics[width=\textwidth]{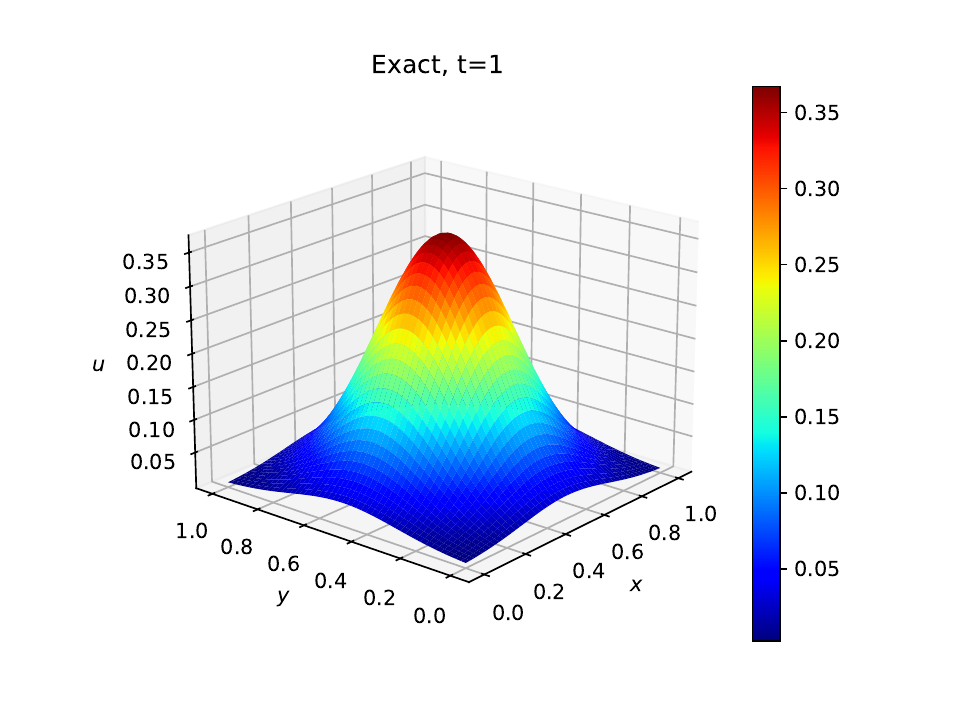}
    \caption{Exact}
    \label{Ex7bb_exact1}
\end{subfigure}
\hfill
\begin{subfigure}[b]{0.4\textwidth}
    \includegraphics[width=\textwidth]{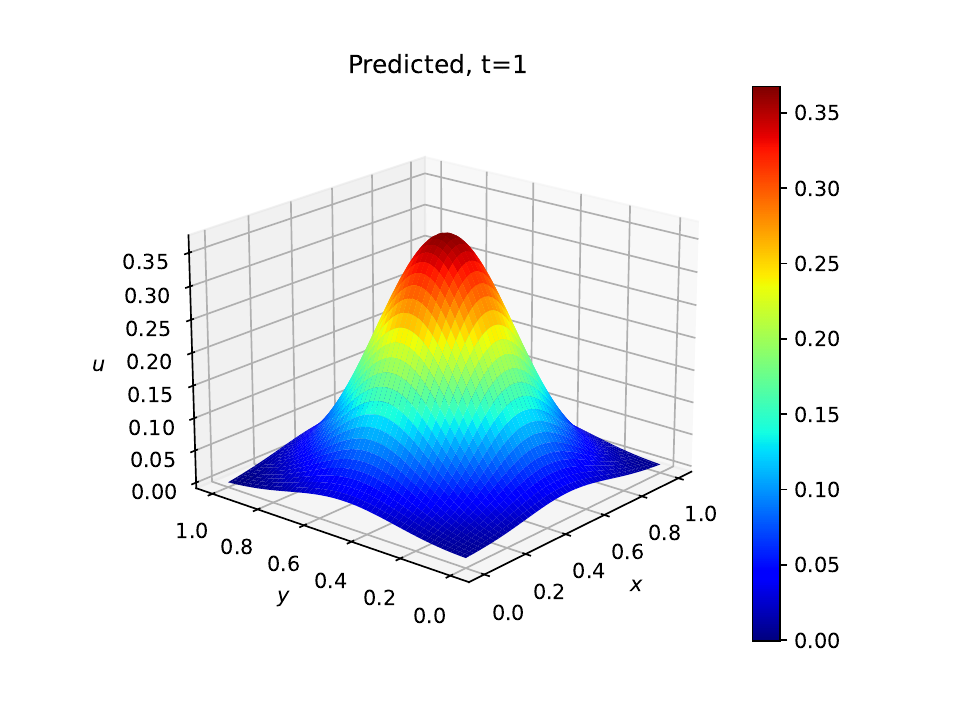}
    \caption{Predicted }
    \label{Ex7bbpre}
\end{subfigure}

\caption{The exact and predicted solutions at \( T = 1 \), with  \( \gamma = 0.0001 \) and \( \beta = 0.1 \).}
\label{fig:Case7b}
\end{figure}

\begin{figure}[htbp]
\centering
\begin{subfigure}[b]{0.4\textwidth}
    \includegraphics[width=\textwidth]{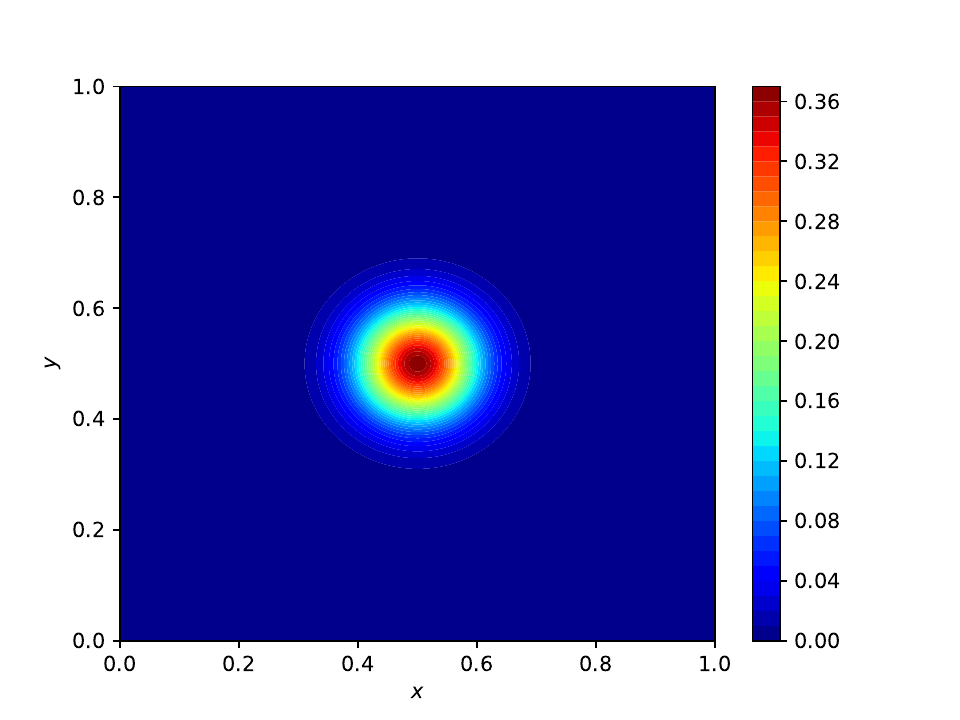}
    \caption{Exact}
    \label{Ex7c_exact1}
\end{subfigure}
\hfill
\begin{subfigure}[b]{0.4\textwidth}
    \includegraphics[width=\textwidth]{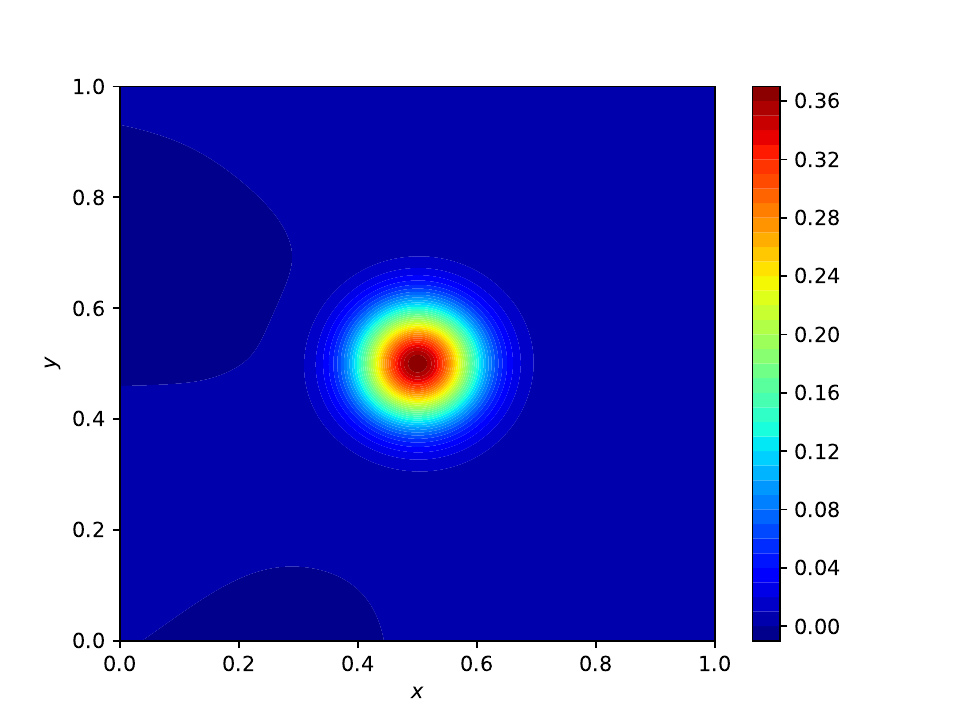}
    \caption{Predicted}
    \label{Ex7cpre}
\end{subfigure}

\begin{subfigure}[b]{0.4\textwidth}
    \includegraphics[width=\textwidth]{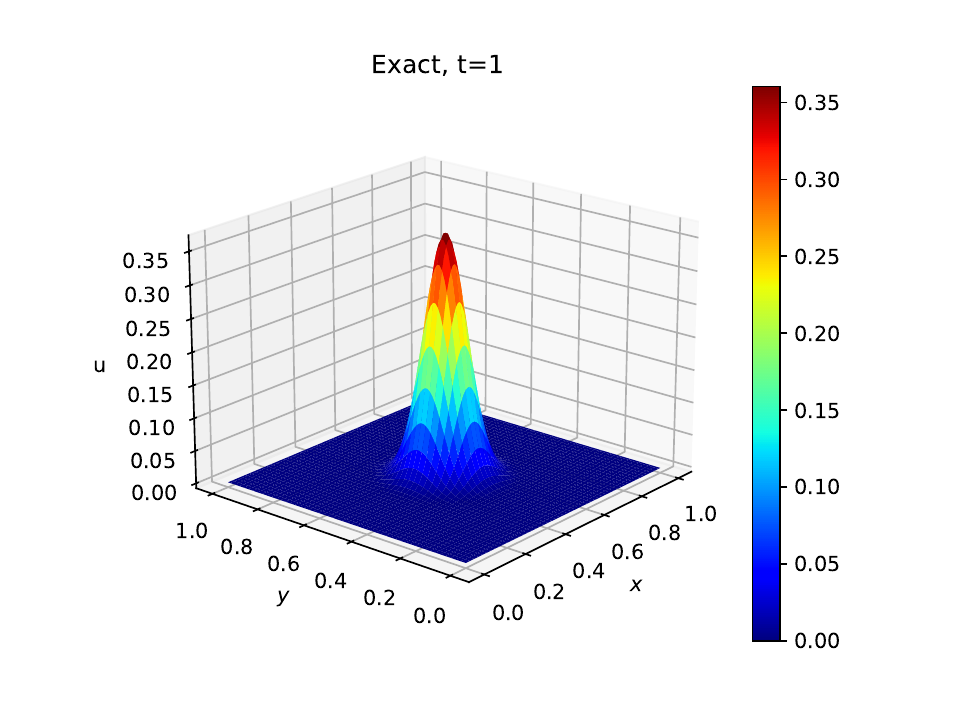}
    \caption{Exact}
    \label{Ex7cc_exact1}
\end{subfigure}
\hfill
\begin{subfigure}[b]{0.4\textwidth}
    \includegraphics[width=\textwidth]{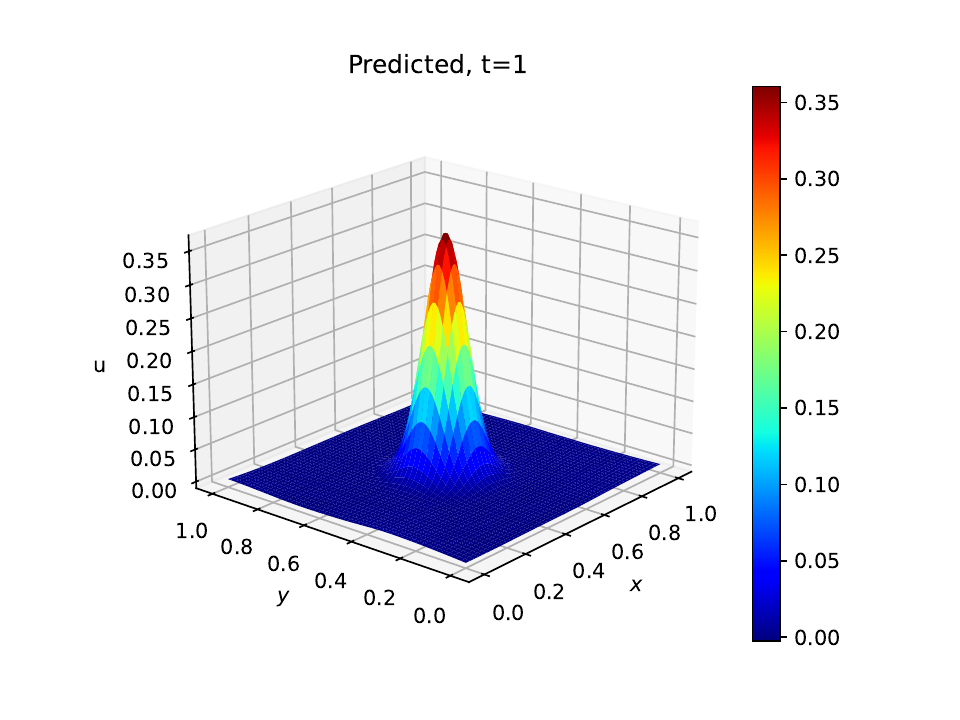}
    \caption{Predicted}
    \label{Ex7cc_predicted1}
\end{subfigure}

\caption{The exact and predicted solutions at \( T = 1 \), with \( \gamma = 0.0001 \) and \( \beta = 0.01. \)
}
\label{fig:Case7c}
\end{figure}

\begin{table}[!ht]
\centering
\begin{tabular}{||c  c c c c  c c c||}
\hline
Methods & $\beta$ & $N$ & $K-1$ & $\bar{d}$ & $\mathcal{E}_{T}$ & $ \mathcal{E}_G$ & \\ [0.5ex] 
\hline
PINN & 1 & 12288 & 4 & 20 &  0.00046 & 2.6e-04 &\\ 
\hline
RWa PINN & 1 & 12288 & 4 & 20  & 0.00022 & 1.0e-04 &\\ 
\hline
RWb PINN & 1 &  12288 & 4 & 20  & 0.00023 & 9.4e-05 & \\ 
\hline
\end{tabular}
\caption{PINN and RW PINN Configuration for Section \ref{Example7}.}
\label{table_8}
\end{table}
\section{Discussions}\label{sec:4} 
This section presents the statistical analysis of the first and last experiments, conducted using RStudio software. Figures \ref{fig:ex2a_errors_analysis} and \ref{fig:ex7a_errors_analysis} provide an overall comparison.  Figure \ref{fig:Ex_2a_Training_error} presents the training error trends for neuron counts of 12, 16, 20, and 24, evaluated over 500, 1000, and 5000 L-BFGS iterations for the methods examined. Additionally, Figure \ref{fig:Ex2a_bar_plot} presents a bar plot comparing the error \( \mathcal{E}_G \) and relative  errors  \( \mathcal{E}^{r}_G \) for the best-performing configuration, corresponding to the maximum LBFGS iterations (5000). These visualizations provide insights into training behavior, error convergence, and the impact of hyperparameter selection. Similarly, Figures \ref{fig:Ex_7a_Training_error} depict the variation in training error for the last experiment, considering neuron counts of 20, 24, 28 and 32 over the same LBFGS iterations. Figure \ref{fig:Ex7a_bar_plot} presents a bar plot comparing the error \( \mathcal{E}_G \) and relative  errors  \( \mathcal{E}^{r}_G \) for the best-performing configuration at 5000 LBFGS iterations. These visualizations facilitate the analysis of training behavior, error convergence, and the influence of hyperparameter selection across experiments. The analysis highlights how the number of neurons affects different error types. Increasing the LBFGS iterations leads to error reduction.

\begin{figure}[htbp]
\centering
\begin{subfigure}[b]{0.95\textwidth}
    \includegraphics[width=\textwidth]{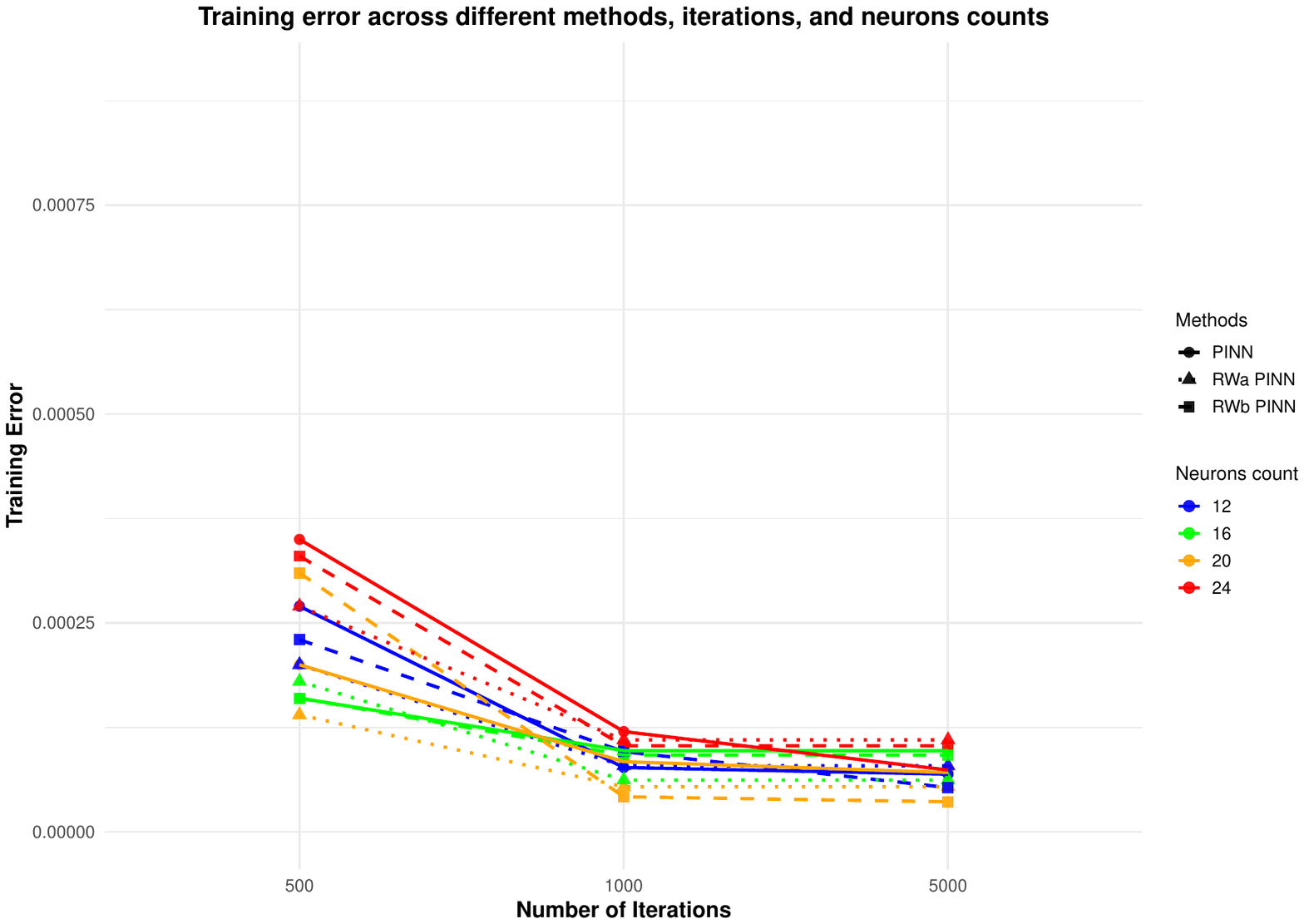}
    \caption{Training error of \ref{Example2} with varying neuron counts and different L-BFGS iterations.}
    \label{fig:Ex_2a_Training_error}
\end{subfigure}

\begin{subfigure}[b]{0.9\textwidth}
    \includegraphics[width=\textwidth]{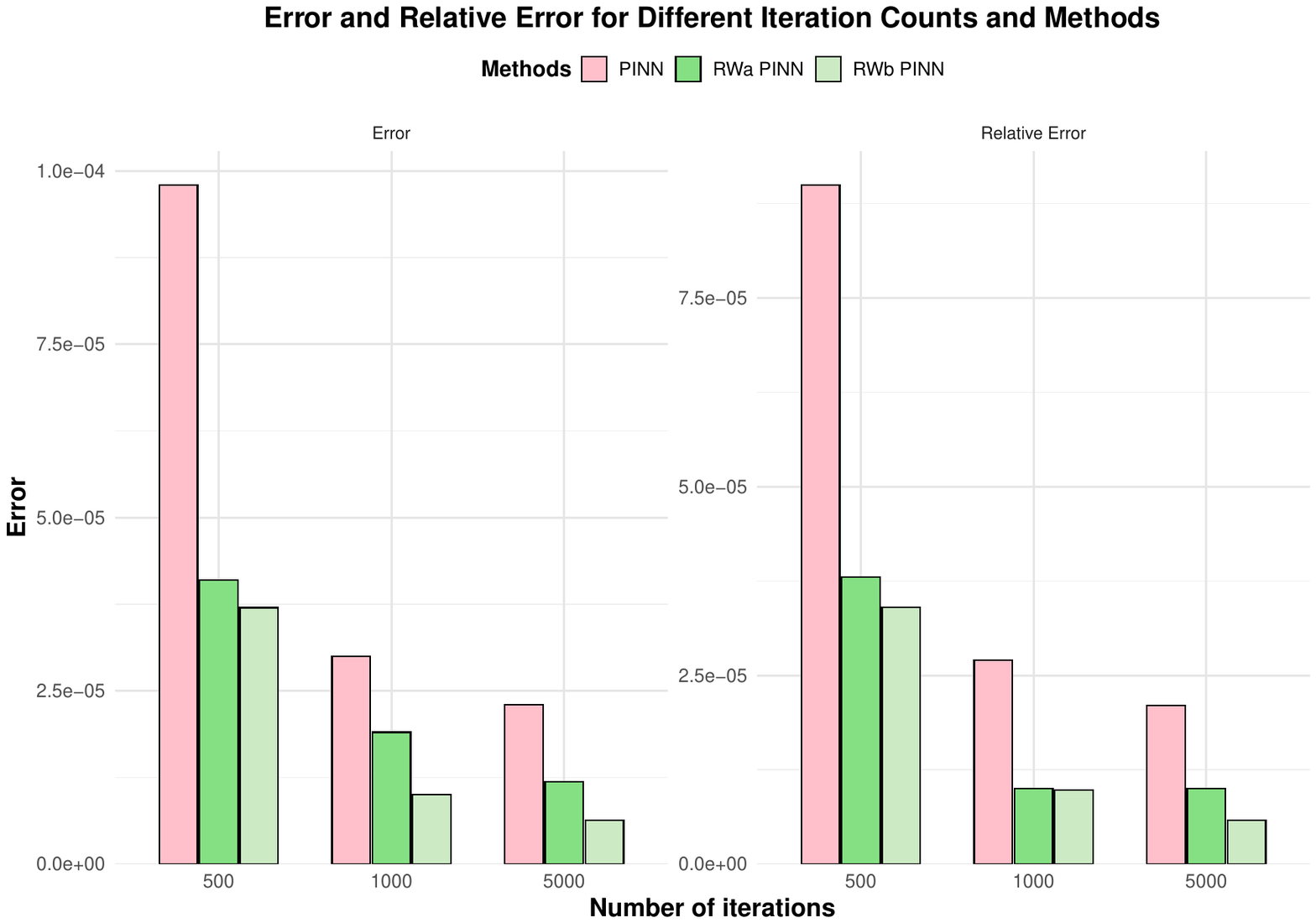}
    \caption{Bar plot of error \( \mathcal{E}_G \) and relative error \( \mathcal{E}^{r}_G \)  for numerical experiment \ref{Example2} with different L-BFGS iterations under the different hyperparameter configuration.}
    \label{fig:Ex2a_bar_plot}
\end{subfigure}
\caption{Statistical measure of errors for numerical experiment \ref{Example2}.
}
\label{fig:ex2a_errors_analysis}
\end{figure}
\begin{figure}[htbp]
\centering
\begin{subfigure}[b]{0.9\textwidth}
    \includegraphics[width=\textwidth]{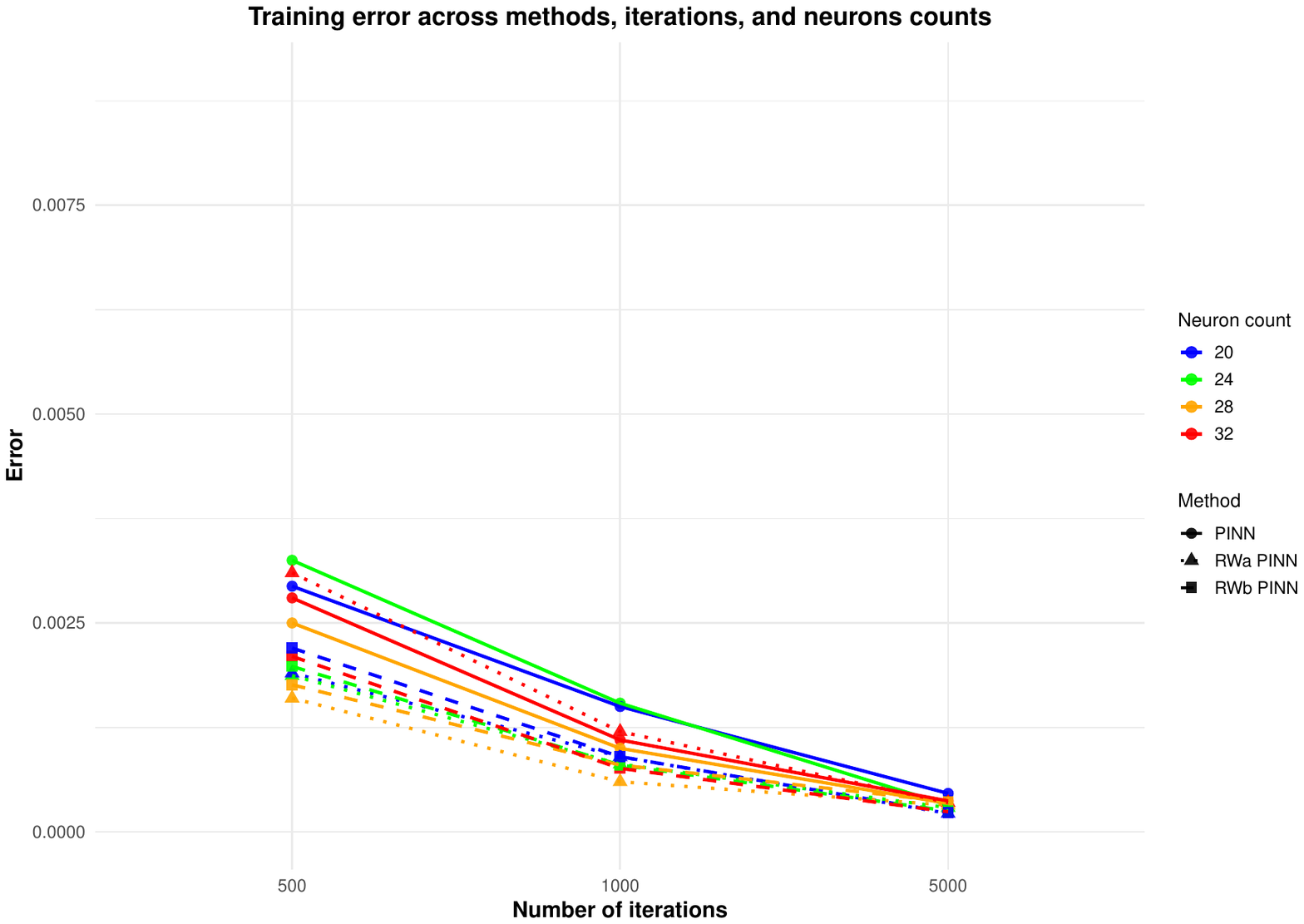}
    \caption{Training error of numerical experiment \ref{Example7} for \(\beta=1\) with different neuron counts and varying L-BFGS iterations}
    
    \label{fig:Ex_7a_Training_error}
\end{subfigure}

\begin{subfigure}[b]{0.9\textwidth}
    \includegraphics[width=\textwidth]{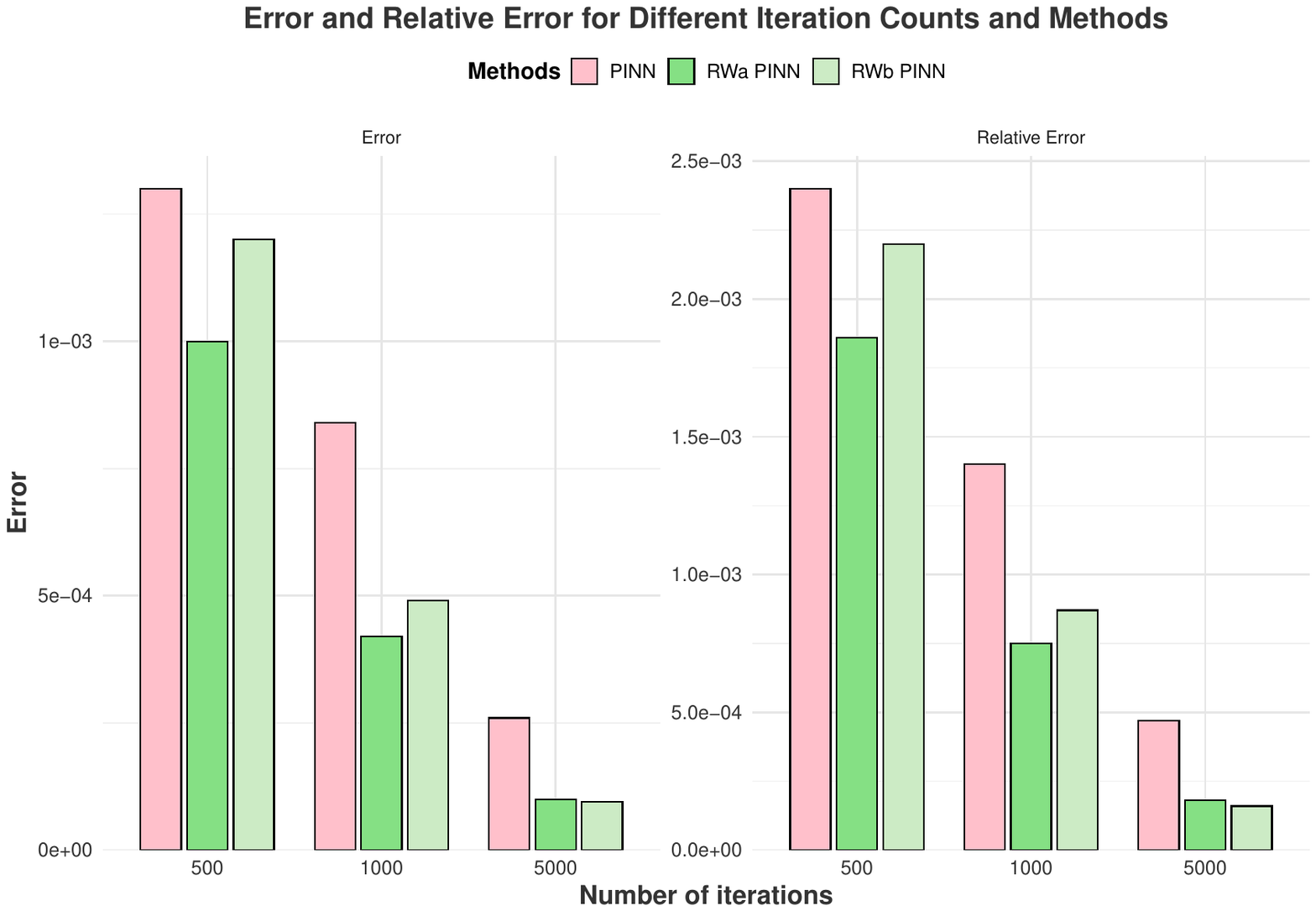}
    \caption{Bar plot of errors \( \mathcal{E}_G \)  and relative errors \( \mathcal{E}^{r}_G \)  for numerical experiment \ref{Example7} with \( \beta = 1 \), using different L-BFGS iterations under the different hyperparameter configuration.
}
    \label{fig:Ex7a_bar_plot}
\end{subfigure}
\caption{Statistical measure of errors for numerical experiment \ref{Example7} with \( \beta = 1 \).}
\label{fig:ex7a_errors_analysis}
\end{figure}
\section{Conclusion}\label{sec:5}
In numerous simulations, the proposed method demonstrates improved accuracy in nonlinear cases. The approach is also applicable to other residual-based methods, where it can further enhance accuracy compared to existing techniques. Consequently, the proposed framework can be regarded as one of the more effective residual-based methodologies. The approach involves training a neural network to approximate classical solutions by minimizing the residuals of the governing PDE. Both well-posed problems (with initial and boundary conditions) and ill-posed problems (without complete initial and boundary data) are considered, using a gradient-based optimization method. For consistency across models, the numerical experiments employ the same training data size, network architecture (number of hidden neurons per layer), and optimization scheme. Theoretical error bounds for the PINN approximation are derived, and both forward and inverse numerical experiments are conducted to demonstrate the effectiveness of RW-PINNs and PINNs in solving nonlinear PDEs efficiently.
 The RW-PINN approach is demonstrated on the nonlinear reaction–diffusion equations to obtain accurate numerical solutions. A comparative analysis with existing PINN-based methods demonstrates that the proposed RW-PINN algorithms effectively address both forward and inverse problems for the nonlinear reaction–diffusion equations. Rigorous error estimates for PINNs are derived, and extensive numerical experiments are conducted to evaluate their performance in solving nonlinear equations. Additionally, we establish the convergence and stability of the neural network.

\section*{Appendix}

The generalization error estimate for the given equation is derived for the forward problem under the assumption that $\mathrm{W_{RWa}} = \mathrm{W_{RWb}} = 1$.

\begin{appendixthm}\label{thm:0}
\label{thm:bergess}
Let \( u \in C^2( \mathbf{D_1} \times [0,T]) \) be the unique classical solution of the Burgess equation \eqref{burgis2}, where the source term \( R \) satisfies Lipchitz condition  (\ref{appendix:buress_lip}). Consider \( u^{\ast} = u_{\theta^{\ast}} \), a PINN approximation obtained through Algorithm \ref{alg1}, corresponding to the loss functions \eqref{eqn:La} and \eqref{eqn:Lb}. 
Then, the generalization error \eqref{eq:gen} satisfies the following bound:
\begin{equation}
    \label{eq:hegenb}
    \boldsymbol{\mathcal{E}}_{G} \leq C_1 \left(\boldsymbol{\mathcal{E}}_T^{tb}+\boldsymbol{\mathcal{E}}_T^{int}+C_2(\boldsymbol{\mathcal{E}}_T^{sb})^{\frac{1}{2}} + (C_{quad}^{tb})^{\frac{1}{2}}N_{tb}^{-\frac{\alpha_{tb}}{2}} +  (C_{quad}^{int})^{\frac{1}{2}}N_{int}^{-\frac{\alpha_{int}}{2}} + C_2  (C_{quad}^{sb})^{\frac{1}{4}}N_{sb}^{-\frac{\alpha_{sb}}{4}} \right),
\end{equation}
where the constants are given by:
\begin{equation}
    \label{eq:hct}
    \begin{aligned}
        C_1 &= \sqrt{T + (1+2C_R)T^2e^{(1+2C_R)T}}, \quad C_2 = \sqrt{C_{\partial \mathbf{D_1}}(u,u^{\ast})T^{\frac{1}{2}}}, \\
        C_{\partial \mathbf{D_1}} &=0.5|\partial \mathbf{D_1}|^{\frac{1}{2}}\left(\|u\|_{C^1([0,T] \times \partial \mathbf{D_1})} + \|u^{\ast}\|_{C^1([0,T] \times \partial \mathbf{D_1})}\right).
    \end{aligned}
\end{equation}
The constants 
\( C_{quad}^{tb} = C_{quad}^{tb}(\|\mathfrak{R}^2_{tb,\theta^{\ast}}\|_{C^2}) \), 
             \( C_{quad}^{sb} = C_{quad}^{tb}(\|\mathfrak{R}^2_{sb,\theta^{\ast}}\|_{C^2}) \), 
             and 
             \( C_{quad}^{int} = C_{quad}^{int}(\|\mathfrak{R}^2_{int,\theta^{\ast}}\|_{C^{0}}) \)
arise from the quadrature error.

\end{appendixthm}

\begin{proof}
From equation~\eqref{eq:main_res}, the interior residual is defined as
\begin{equation}
\begin{aligned}
\mathfrak{R}_{\text{int},\theta} = \mathfrak{R}_{\text{int},\theta}(t,x), 
\quad \forall (t,x) \in [0,T] \times \mathbf{D_1}.
\end{aligned}
\end{equation} 

For the Burgess equation, the PDE residual can be  written as
\begin{equation}\label{eqn:Rab_mod}
\begin{aligned}
\mathfrak{R}_{\text{int},\theta} 
= \partial_{t} u_{\theta} - \tfrac{1}{2}\,\partial_{xx} u_{\theta} - R(u_{\theta}), 
\quad \forall (t,x) \in [0,T] \times \mathbf{D_1}.
\end{aligned}
\end{equation}

The residuals corresponding to the initial and boundary conditions are formulated as
\begin{equation}\label{eqn:ber_proof}
\begin{aligned}
\mathfrak{R}_{\text{tb}} &= u_{\theta} - u_{0}, 
\quad \forall x \in \mathbf{D_1}, \\
\mathfrak{R}_{\text{sb}} &= u_{\theta} - u_{b}, 
\quad \forall (t, x) \in \partial \mathbf{D_1}.
\end{aligned}
\end{equation}

Now, the theorem can be proved by following the approach of~\cite{mishra2023estimates}\cite{DeRyck2024}.
\end{proof}

\begin{appendixthm}\label{thm:1}
Consider \( u \in C^{4}\left( [0,1] \times [0,T] \right) \) as the unique solution to the EFK \eqref{eq:eqmain}. Let \( u^{\ast} = u_{\theta^{\ast}} \) be the Physics-Informed Neural Network (PINN) approximation obtained using Algorithm \ref{alg1}. The non linear term \(F(u)\) satisfied the Lipchitz condition \ref{appendix:EEK_lip}. Then, the generalization error \eqref{eq:gen} satisfies the following bound:
\begin{equation}
    \label{eq:hegenb}
    \boldsymbol{\mathcal{E}}_{G} \leq C_1 \left(\boldsymbol{\mathcal{E}}_T^{tb}+\boldsymbol{\mathcal{E}}_T^{int}+C_2(\boldsymbol{\mathcal{E}}_T^{sb})^{\frac{1}{2}} + (C_{quad}^{tb})^{\frac{1}{2}}N_{tb}^{-\frac{\alpha_{tb}}{2}} +  (C_{quad}^{int})^{\frac{1}{2}}N_{int}^{-\frac{\alpha_{int}}{2}} + C_2  (C_{quad}^{sb})^{\frac{1}{4}}N_{sb}^{-\frac{\alpha_{sb}}{4}} \right),
\end{equation}
where:
\begin{itemize}
    \item \( \boldsymbol{\mathcal{E}}_{G} \) denotes the generalization error.
    \item \( \boldsymbol{\mathcal{E}}_{T}^{tb} \), \( \boldsymbol{\mathcal{E}}_{T}^{sb} \), and \( \boldsymbol{\mathcal{E}}_{T}^{int} \) are the errors associated with the temporal boundary, spatial boundary, and interior points, respectively.
    \item The constants are defined as:
    \begin{align*}
        C_1 &=\sqrt{(T + 2T^2 C_{3} e^{2C_{3}T})}, \\
        C_3 &= \sqrt{\gamma (\| u \|_{C^2_x} + \| u^{\ast} \|_{C^2_x})^2 + (\| u \|_{C^2_x} + \| u^{\ast} \|_{C^2_x}) + \frac{1}{2} + C_E}, \\
        C_2 &= \sqrt{ 8 \gamma (\| u \|_{C^3_x} + \| u^{\ast} \|_{C^3_x}) T^{1/2} }.
    \end{align*}
    \item The quadrature error constants are:
    \begin{align*}
        C_{quad}^{tb} &= C_{quad}^{tb}(\|\mathfrak{R}^2_{tb,\theta^{\ast}}\|_{C^4}), \\
        C_{quad}^{sb} &= C_{quad}^{sb}(\|\mathfrak{R}^2_{sb,\theta^{\ast}}\|_{C^2}), \\
        C_{quad}^{int} &= C_{quad}^{int}(\|\mathfrak{R}^2_{int,\theta^{\ast}}\|_{C^{0}}).
    \end{align*}
    \item \( N_{tb} \), \( N_{sb} \), and \( N_{int} \) represent the number of training points for the temporal boundary, spatial boundary, and interior domain, respectively.
\end{itemize}
\end{appendixthm}

\begin{proof}
Let $\bar{u} :u^{\ast} - u $. \\
We can write the residual of EFK equation \eqref{eq:eqmain} in the following form:
\begin{equation}\label{eqn:R_EFK}
\bar{u}_{t} + \gamma \bar{u}_{xxxx} - \bar{u}_{xx} + F(u^{\ast} -u) = \mathfrak{R}_{int},   \quad \forall (t,x) \in [0,T] \times \mathbf{D_2}
\end{equation}
\[ \text{where}~ x \in  (0, 1) ~ \text{and}  ~  t \in (0, T),\]
\[ \bar{u}(x,0) = \mathfrak{R}_{tb} ~\text{where}~ x \in  (0, 1),\]
\[ \bar{u}(0,t) = \mathfrak{R}_{sb_{1}} ~\text{where}~ t \in  (0, T),\]
\[ \bar{u}(1,t) = \mathfrak{R}_{sb_{2}} ~\text{where}~ t \in  (0, T),\]
\[ \bar{u}_{xx}(0,t) = \mathfrak{R}_{sb_{3}} ~\text{where}~ t \in  (0, T),\]
\[ \bar{u}_{xx}(1,t) = \mathfrak{R}_{sb_{4}} ~\text{where}~ t \in  (0, T).\]
We can write 
\begin{equation}\label{eq:assum}
\bar{u}\bar{u}_{xxxx} = (\bar{u}\bar{u}_{xxx})_x - (\bar{u}_x \bar{u}_{xx})_x + (\bar{u}_{xx})^2
\end{equation}
We denote $ \mathfrak{R}_{int} = \mathfrak{R}_{int, \theta^{\ast}}, \mathfrak{R}_{tb} = \mathfrak{R}_{tb, \theta^{\ast}}, \mathfrak{R}_{sb_{i}} = \mathfrak{R}_{sb, \theta^{\ast}} ~ \text{where}~ i = 0 ~ \text{to} ~4$ 
Multipyling the eq \eqref{eqn:R_EFK} by $\bar{u}$ and integrating over $(0,1),$ we get

\begin{align}
\frac{1}{2}\frac{d}{dt} \int_{0}^{1} \bar{u}^{2} \, dx + \gamma \int_{0}^{1} \bar{u} \bar{u}_{xxxx} \, dx - \int_{0}^{1} \bar{u} \bar{u}_{xx} \, dx  + \int_{0}^{1} \bar{u}F(u^{\ast}-u) \, dx = \int_{0}^{1} \bar{u}\mathfrak{R}_{int}  \, dx. 
\end{align}
\begin{align}
\frac{1}{2}\frac{d}{dt} \int_{0}^{1} \bar{u}^{2} \, dx &= - \gamma \int_{0}^{1} \bar{u} \bar{u}_{xxxx} \, dx + \int_{0}^{1} \bar{u} \bar{u}_{xx} \, dx  - \int_{0}^{1} \bar{u}F(u^{\ast}-u) \, dx + \int_{0}^{1} \mathfrak{R}_{int} \bar{u} \, dx. 
\end{align}
Now putting the value of  $\bar{u}\bar{u}_{xxxx}$ from \eqref{eq:assum}, we get
\begin{align}
\frac{1}{2}\frac{d}{dt} \int_{0}^{1} \bar{u}^{2} \, dx  = -\gamma \bar{u}\bar{u}_{xxx}\mid_{0}^{1} + \gamma \bar{u}_{x}\bar{u}_{xx}\mid_{0}^{1} - \gamma\int\limits_{0} ^{1}  \bar{u}_{xx}^{2}dx+\int\limits_{0}^{1}\bar{u}\bar{u}_{xx}dx+ \int_{0}^{1} \bar{u}F(u^{\ast}-u) \, dx   + \int_{0}^{1} \mathfrak{R}_{int} \bar{u} \, dx.
\end{align}

\begin{align}
\frac{1}{2} \frac{d}{dt} \int_{0}^{1} \bar{u}^{2} \, dx  
&\leq \gamma \|\bar{u}\|_{C_{x}^{3}} 
\left( |\mathfrak{R}_{sb_{1}}| + |\mathfrak{R}_{sb_{2}}| + |\mathfrak{R}_{sb_{3}}| + |\mathfrak{R}_{sb_{4}}| \right)  + \gamma \left( \|\bar{u}\|_{C^2_x} \right)^{2} \int_{0}^{1} 1 \,dx  
+ \|\bar{u}\|_{C^2_x} \int_{0}^{1} \bar{u} \, dx  \notag \\ 
&\quad + \frac{1}{2} \int_{0}^{1} \mathfrak{R}_{int}^{2} \,dx  
+ \left( \frac{1}{2} + C_E \right) \int_{0}^{1} \bar{u}^{2} \, dx.
\end{align}
\begin{align}
\frac{1}{2}\frac{d}{dt} \int_{0}^{1} \bar{u}^{2} \, dx  \leq \gamma  \parallel \bar{u} \parallel_{C_{x}^{3}}\left( \vert \mathfrak{R}_{sb_{1}} \vert + \vert \mathfrak{R}_{sb_{2}} \vert +  \vert \mathfrak{R}_{sb_{3}} \vert + \vert \mathfrak{R}_{sb_{4}} \vert \right)  + \left( \gamma (\|\bar{u}\|_{C^2_x})^2 + \|\bar{u}\|_{C^2_x}+ \frac{1}{2}+ C_E \right) \left( \int_{0}^{1} \bar{u}^2 dx  \right) + \frac{1}{2}\int_{0}^{1} \mathfrak{R}_{int}^{2} dx.
\end{align}
\begin{align}
\frac{1}{2} \frac{d}{dt} \int_{0}^{1} \bar{u}^{2} \, dx  
&\leq \gamma (\| u \|_{C^3_x} + \| u^{\ast} \|_{C^3_x}) \left( \vert \mathfrak{R}_{sb_{1}} \vert + \vert \mathfrak{R}_{sb_{2}} \vert +  \vert \mathfrak{R}_{sb_{3}} \vert + \vert \mathfrak{R}_{sb_{4}} \vert\right) \nonumber \\
&\quad + \left( \gamma (\| u \|_{C^2_x} + \| u^{\ast} \|_{C^2_x})^2 + (\| u \|_{C^2_x} + \| u^{\ast} \|_{C^2_x} )+ \frac{1}{2} + C_E \right) \int_{0}^{1} \bar{u}^2 \, dx \nonumber \\
&\quad + \frac{1}{2} \int_{0}^{1} \mathfrak{R}_{int}^{2} \, dx.
\end{align}

\begin{align}
\frac{1}{2}\frac{d}{dt} \int_{0}^{1} \bar{u}^{2} \, dx \leq C_{1} \sum\limits_{i}^{4} \left(   \mathfrak{R}_{sb_{i}} \right)   + C_{2} \int_{0}^{1} \bar{u}^2 dx  \quad & +  \frac{1}{2} \int_{0}^{1} \mathfrak{R}_{int}^{2} \, dx. 
\end{align}
The mixed norm is defined as  
\[
\| u \|_{C^m_t C^n_x}.
\]
Then, integrating the above inequality over \([0, \bar{T}]\) for any \(\bar{T} \leq T\), and using the Cauchy-Schwarz and Gronwall’s inequalities, we obtain the following estimate:
\begin{align}
\int\limits_{0}^{1} \bar{u}^{2} \, dx 
&\leq \int\limits_{0}^{1} \mathfrak{R}_{tb}^{2} \, dx  
+ 2C_{1} \bar{T}^{\frac{1}{2}} \sum\limits_{i=1}^{4}  
\left(\int\limits_{0}^{T} \mathfrak{R}^{2}_{sb_{i}} \, dt \right)^{\frac{1}{2}}  \notag \\  
&\quad    
+ \int\limits_{0}^{T} \int\limits_{0}^{1} \mathfrak{R}_{int}^{2} \, dx \, dt + 2C_{2} \int\limits_{0}^{T} \int\limits_{0}^{1} \bar{u}^{2} \, dx \, dt.
\end{align}
\begin{align}
\int\limits_{0}^{1} \bar{u}^{2} \, dx &\leq (1+2TC_{2}e^{2C_{2}T})\left[ \int\limits_{0}^{1} \mathfrak{R}_{tb}^{2} dx + 8C_{1}T^{\frac{1}{2}} \sum\limits_{i}^{4} \left(\int\limits_{0}^{T} \mathfrak{R}^{2}_{sb_{i}}dt \right)^{\frac{1}{2}}    \right] \notag\\   \quad & + (1+2\bar{T}C_{2}e^{2C_{2}T})\left[2\int\limits_{0}^{T}\int\limits_{0}^{1} \mathfrak{R}_{int}^{2} dxdt  + 2C_{2}\int\limits_{0}^{T}\int\limits_{0}^{1} \bar{u}^{2} dxdt\right].
\end{align}
Again integrating over $T$ we get,

\begin{align}
\boldsymbol{\mathcal{E}}_{G} &\leq (T+2T^2C_{2}e^{2C_{2}T})\left[ \int\limits_{0}^{1} \mathfrak{R}_{tb}^{2} dx + 8C_{1}T^{\frac{1}{2}} \sum\limits_{i}^{4} \left(\int\limits_{0}^{T} \mathfrak{R}^{2}_{sb_{i}}dt \right)^{\frac{1}{2}}    \right] \notag\\   \quad & + (\bar{T}+2T^2C_{2}e^{2C_{2}T})\left[2\int\limits_{0}^{\bar{T}}\int\limits_{0}^{1} \mathfrak{R}_{int}^{2} dxdt  + 2C_{2}\int\limits_{0}^{T}\int\limits_{0}^{1} \bar{u}^{2} dxdt\right],
\end{align}
where $C_{1}=\left( \gamma (\| u \|_{C^3_x} + \| u^{\ast} \|_{C^3_x})  \right)$ and $C_{2} = \left( \gamma (\| u \|_{C^2_x} + \| u^{\ast} \|_{C^2_x})^2 + (\| u \|_{C^2_x} + \| u^{\ast} \|_{C^2_x} )+ \frac{1}{2} + C_E \right)$ .   
\end{proof}

\begin{appendixthm}[Lemma] \label{appendix:EFK_uniqe} 
Under Assumption (H1), if \( u_0 \in H^2(\mathbf{D_2}) \), then the EFK equation \eqref{eq:eqmain} admits a unique solution \( u \) on \( [0,T] \) satisfying  
\[
    u \in C([0,T]; H^2(\mathbf{D_2})) \cap L^2(0,T; H^4(\mathbf{D_2})).
\]  
\end{appendixthm}  

\begin{appendixthm}[Lemma] \label{appendix:buress_unique} 
Under Assumption (H2), let \( u_0 \in H^1(\mathbf{D_1}) \), then the Burgess equation \eqref{eq:eqmain} has a unique solution \( u \) on \( [0,T] \) such that
\[
    u \in C([0,T]; H^1(\mathbf{D_1})) \cap L^2([0,T]; H^2(\mathbf{D_1})).
\]  
\end{appendixthm}  

\begin{appendixthm}[Lemma] \label{appendix:buress_lip}
Assuming that the non-linearity is globally Lipschitz, there exists a constant \( C_R \) (independent of \( u_{1},u_{2} \)) such that  
\begin{equation}
    \label{eq:Bur_lip1}
    |R(u_{1}) - R(u_{2})| \leq C_R|u_{1}-u_{2}|, \quad u_{1},u_{2} \in \mathbb{R}.
\end{equation}
\end{appendixthm}  

\begin{appendixthm}[Lemma] \label{appendix:EEK_lip}
Let \( u \in C^{\infty}(\mathbf{D_2}) \), where \( \mathbf{D_2} \) is a closed set in \( \mathbb{R}^d \). Consider the function \( F(u) = u^3 - u \). Then, \( F(u) \) satisfies a Lipschitz condition, i.e., there exists a constant \( C_{F} \) (independent of \( u_{1},u_{2} \)) such that:
\[
|F(u_1) - F(u_2)| \leq  C_{F}|u_{1} - u_{2}|
\]
for all \( u_1, u_2 \in \mathbf{D_2} \).
\end{appendixthm}  

\begin{proof}
See \cite{ilati2020analysis}.
\end{proof}
\begin{appendixthm}[Theorem \cite{danumjaya2006numerical}] \label{appendix:uniform_bound}
Let \( u_0 \in H^2_0(\mathbf{D_2}) \) be the initial condition of \( u(t) \), satisfying  
\begin{equation}
    u(0) = u_0.
\end{equation}
There exists a constant \( C > 0 \) such that the following bound holds for all \( t > 0 \):  
\begin{equation}
    \| u(t) \|_{L^2} \leq C(\gamma, \| u_0 \|_{L^2}).
\end{equation}
Moreover, the solution remains uniformly bounded in the \( L^\infty \)-norm as  
\begin{equation}
    \| u(t) \|_{L^\infty} \leq C(\gamma, \| u_0 \|_{L^2}), \quad t > 0.
\end{equation}
\end{appendixthm}
\subsection*{Declaration}
\subsubsection*{Declaration of competing interest} The authors declare that they have no competing interests.
\subsubsection*{Conflict of interest} The authors declare that they have no conflict of interest.
\subsection*{Funding} No funding was received for this study.
\section*{Acknowledgment}
The first author acknowledges the Ministry of Human Resource Development (MHRD), Government of India, for providing institutional funding and support at IIT Madras.

\bibliographystyle{abbrv}
\bibliography{sample}
\end{document}